\documentclass[10pt]{amsart}

\usepackage{latexsym}
\usepackage{amssymb}
\usepackage{amsmath}
\usepackage{amsfonts}
\usepackage{textcomp}
\usepackage{amscd}
\usepackage{amsthm}
\usepackage{amsxtra}
\usepackage[all]{xy}
\usepackage[dvips]{graphicx}

\numberwithin{equation}{section}
\newtheorem{theorem}{Theorem}[section]

\newtheorem{proposition}[theorem]{Proposition}
\newtheorem{lemma}[theorem]{Lemma}

\newtheorem{corollary}[theorem]{Corollary}

\newtheorem{conjecture}[theorem]{Conjecture}

\theoremstyle{remark}
\newtheorem{remark}{Remark}[section]

\theoremstyle{definition}

\begin{document}

\title[Framed deformation of Galois representation]
{Framed deformation of Galois representation}

\author{Lin Chen}
\address{Department of Mathematics\\
University of California,Los Angeles}
\email{chenlin@math.ucla.edu}

\begin{abstract}
We studied framed deformations of two dimensional Galois
representation of which the residue representation restrict to
decomposition groups are scalars, and established a modular
lifting theorem for certain cases. We then proved a family version
of the result, and used it to determine the structure of
deformation rings over characteristic zero fields. As a corollary,
we obtain the $\mathcal{L}$-invariant of adjoint square
representation associated to a Hilbert Hecke eigenform.
\end{abstract}

\maketitle

\setcounter{tocdepth}{5} \setcounter{page}{1}

\section{Introduction}
Given a Hilbert modular Hecke eigenform $f$ (over some totally
real field $F$), one can associated a two dimensional continuous
Galois representation $\rho_f$ (into $GL_2(K)$ for some finite
extension $K$ over $\mathbb{Q}_p$). When $f$ is ordinary, $\rho_f$
is nearly ordinary at $p$. Conversely, assume the residue
representation of a given representation $\rho$ is modular and
nearly ordinary, and assume it satisfies a technical
distinguishedness condition, the representation can be proved to
be modular by Taylor-Wiles \cite{TW}, Fujiwara \cite{Fuj} and
Skinner-Wiles \cite{SW}. The main idea is to study certain types
of deformation problems, and identify the universal deformation
ring to the localization of the Hecke algebra at the maximal ideal
determined by $f$. This type of result is often called ``$R=T$''
theorem.

When the Distinguishedness condition of a nearly ordinary
representation fails, the deformation functor with the prescribed
local conditions is no longer representable. This phenomena also
appeared in the deformation of Barsotti-Tate representations
studied by Kisin \cite{Kis}. To make a representable functor,
instead consider the deformations with prescribed local
conditions, Kisin consider the functor associated to each ring a
deformation of Barsotti-Tate representation, together with a basis
lifting a fixed chosen basis of the residue representation. The
basis eliminate automorphisms of the functor, and thus form a
representable functor. As a result, Kisin proved certain framed
version of the ``$R=T$'' theorem, up to some finite torsion due to
the Barsotti-Tate condition, which is enough to prove the
modularity.

In this paper, we will consider the local condition that the
restriction of $\overline{\rho}$ to the decomposition groups are
scalers, so the distinguishedness condition fails. Inspired by
Kisin's work, we invent a deformation ring
$\mathcal{R}_\mathbf{p}^{\vartriangle,\psi,s}$, which represent
the functor associated each ring $A$ a deformation of the
representation of a decomposition group into the Borel subgroup
(upper triangular matrices) together with a basis, which is
transformed to the standard basis by an element of the Borel
subgroup. The reason we make such choice is the Schlessinger
criterion, which ensure the functor is representable. Using these
rings, we proved a framed version of the ``$R=T$'' theorem in
section \ref{sec5} and the modularity assuming the residue
representation is modular in Theorem \ref{Thm6.1}. We then
generalize this result to Hida's family of Hilbert modular forms.
As an application, we prove Hida's conjecture that $\mathcal{R}_K$
is a power series ring, where $\mathcal{R}_K$ is the universal
deformation ring representing the functor of deformations into
representations over Artinian $K$-algebras. This result implies
Hida's conjectural formula \cite{H07d} of the
$\mathcal{L}$-invariant of the adjoint representation, in
Corollary \ref{Cor8.4}.

In section \ref{sec2}, we provide some well known facts on Hilbert
modular forms. In section \ref{sec3}, we study the framed
deformation (without local conditions) of Kisin \cite{Kis}, and
calculate their tangential dimensions using Galois cohomology. The
ring $\mathcal{R}_\mathbf{p}^{\vartriangle,\psi,s}$ and the
deformations of the scaler representation of the decomposition
group into Borel subgroups are studied in section \ref{sec4},
which is the most original part of the paper. Section \ref{sec5}
and \ref{sec6} are the Taylor-Wiles system and the modularity. We
generalize these result to Hida's family in section \ref{sec7} and
obtained the $\mathcal{L}$-invariant in section \ref{sec8}.

Finally, I want to express my grateful to Professor Hida for all
his patience, time and care for the past four years. Without his
educations, suggestions, comments and encouragements, this work
can not be finished. I also want to thank Professor Qingchun Tian
of Peking University for the conversation in the summer of 2007.

\section{Hilbert Modular Forms} \label{sec2}
Fix a totally real number field $F$ of degree $d$. $I$ is the set
of embedding of $\sigma:F\hookrightarrow \overline{\mathbb{Q}}$.
Denote by $\mathbb{A}_{F}$ its ring of adeles, which decompose
into finite and infinite parts as
$\mathbb{A}_{F}=\mathbb{A}_{F}^f\times\mathbb{A}_{F}^{\infty}$.
Let $\overline{\mathbb{Q}}_p$ be the algebraic closure of
$\mathbb{Q}_p$ and $E\subset\overline{\mathbb{Q}}_p$ a finite
extension of $\mathbb{Q}_p$ with integer ring $O_E$. Fix the
embedding
$\mathbb{C}\hookleftarrow\overline{\mathbb{Q}}\hookrightarrow\overline{\mathbb{Q}}_p$
once and for all.

As Hida in \cite{HMF}, we consider the following type of
continuous ``Neben'' characters
$$\varepsilon=(\varepsilon_1,\varepsilon_2:\widehat{\mathcal{O}}^{\times}
\rightarrow\mathbb{C}^{\times},
\varepsilon_+:\mathbb{A}_F/{F}^{\times}\rightarrow
\mathbb{C}^{\times})$$ and the weights
$k=(k_1,k_2)\in\mathbb{Z}[I]^2$ such that $k_1+k_2=(n+1)\cdot I$
for some integer $n$,
$\varepsilon_+|_{\widehat{\mathcal{O}}^{\times}}=\varepsilon_1\varepsilon_2$
and $\varepsilon_+(x_{\infty})=x^{-(k_1+k_2)+I}$. Let $N$ be an
integral ideal of $\mathcal{O}$ and define the $\Gamma_0$ type
congruence subgroup
$$\widehat{\Gamma}_0(N)=\{\begin{pmatrix}a&b\\c&d\end{pmatrix}\in
GL_2(\widehat{\mathcal{O}})|c\in N\widehat{\mathcal{O}}\}$$ and
the $\Gamma_1$ type congruence subgroup
$$\widehat{\Gamma}_1^1(N)=\{\begin{pmatrix}a&b\\c&d\end{pmatrix}\in
GL_2(\widehat{\mathcal{O}})|a-1,b,d-1\in
N\widehat{\mathcal{O}}\}.$$ Let
$\varepsilon^-=\varepsilon_2^{-1}\varepsilon_1$ and assume its
conductor $\mathfrak{c}(\varepsilon^-)\supset N$, then the
character
$$\varepsilon:\widehat{\Gamma}_0(N)\rightarrow \mathbb{C}^{\times}$$
defined by
$\varepsilon\begin{pmatrix}a&b\\c&d\end{pmatrix}=\varepsilon_2(ad-bc)\varepsilon^{-}(a)$
is a continuous character of $\widehat{\Gamma}_0(N)$.

Under the above notation, define the automorphy factor of weight
$k$ as
$$J_k(g,z)=\det(g)^{k_1-I}j(g,z)^{k_2-k_1+I}=\prod_{\sigma\in I}
\det(g_\sigma)^{k_{1,\sigma}}(c_{\sigma}
z_{\sigma}+d_{\sigma})^{k_{2,\sigma}-k_{2,\sigma}+1}$$ for
$g=(g_\sigma)\in GL_2(\mathbb{A}_{F}^{\infty})=GL_2(\mathbb{R})^I$
and $z=(z_\sigma)\in \mathbb{H}^I$, where $\mathbb{H}^I$ is the
$d$ folder upper half plain as usual. Define the Hilbert cusp form
$S_k(N,\varepsilon;\mathbb{C})$ of weight $k$, level $N$ and
``Neben'' type $\varepsilon$ to be the functions $f$ with the
following three conditions:

(A1)\hskip 0.5cm For all $\alpha\in GL_2(F)$, $z\in
Z(\mathbb{A_F})$ and $u\in\widehat{\Gamma}_0(N)C_{\mathbf{i}}$,
where $C_{\mathbf{i}}$ is the stabilizer of
$\mathbf{i}=(\sqrt{-1},\sqrt{-1},\cdots,\sqrt{-1})\in \mathbb{H}^I
$ in $GL_2^+(\mathbb{A}_F^{\infty})$, $f(\alpha
xuz)=\varepsilon_+(z)\varepsilon(u_f)f(x)J_k(u_{\infty},\mathbf{i})^{-1}$.

(A2)\hskip 0.5cm For each $z\in \mathbb{H}^I$, choose an element
$u\in GL_2(\mathbb{A}_F^{\infty})$, then the functions
$f_g:\mathbb{H}^I\rightarrow \mathbb{C}$ defined by
$f_g(z)=f(gu_{\infty})J_k(u_{\infty},\mathbf{i})$ are holomorphic
for all $g\in GL_2(\mathbb{A}_{F}^f)$.

(A3)\hskip 0.5cm
$\int_{\mathbb{A}_{F}/F}f(\begin{pmatrix}1&u\\0&1\end{pmatrix}x)du=0$
for all $x\in GL_2(\mathbb{A}_{F}^f)$.

As usual, one can define the Hecke operators $T_v$ and $S_v$ for
$v\nmid N$ and $U_\ell$ for $\ell|N$ acting on the finite
dimensional complex vector space $S_k(N,\varepsilon;\mathbb{C})$.
Let $\mathcal{W}$ be the integer ring of some number field
containing the values of $\varepsilon_1$,$\varepsilon_2$,
$\varepsilon_+$ and all the conjugates of $\mathcal{O}$ in
$\overline{\mathbb{Q}}$. It's well known that there is a
$\mathcal{W}$ lattice $S_k(N,\varepsilon;\mathcal{W})\subset
S_k(N,\varepsilon;\mathbb{C})$ stable under the action of the
Hecke operators mentioned above such that
$$S_k(N,\varepsilon;\mathcal{W})\otimes_{\mathcal{W}}\mathbb{C}
=S_k(N,\varepsilon;\mathbb{C}).$$

For each $\mathcal{W}$-algebra $A$, define
$S_k(N,\varepsilon;A)=S_k(N,\varepsilon;\mathcal{W})\otimes_{\mathcal{W}}A$.
Define $h_k(N,\varepsilon;A)$ to be the $A$-subalgebra of
$\textrm{End}_A(S_k(N,\varepsilon;A))$ generated by the operators
$T_v$ for $v\nmid N$ and $U_\ell|N$. Then there are isomorphisms
$$\textrm{Hom}_A(S_k(N,\varepsilon;A),A)\cong
h_k(N,\varepsilon;A)$$ and
$$\textrm{Hom}_A(h_k(N,\varepsilon;A),A)\cong
S_k(N,\varepsilon;A)$$ given by the perfect pairing
$(,):h_k(N,\varepsilon;A)\times S_k(N,\varepsilon;A)\rightarrow A$
such that $(h,f)=a(1,f|h)$ for $a(1,f)$ the first coefficient in
the $q$-expansion of $f$.

\vskip 0.5cm

Let $W$ be the completion of $\mathcal{W}$ at a prime above $p$,
so $W$ is a complete discrete valuation ring with residue field
$k$ of characteristic $p$. The following theorem is well known and
should be attributed to many people including Shimura, Deligne,
Serre, Wiles, Blasius, Rogawski and Taylor. The version we present
here is partially adopted from Hida's book \cite{HMF}.

\begin{theorem} \label{ThmHida}
Suppose $k_2-k_1+I\geqslant 2I$. Let $P$ be a prime ideal of
$h=h_k(N,\varepsilon;W)$ and assume that the characteristic of the
fraction field of $h/P$ is different from $2$. Then, there is a
continuous semisimple Galois representation $\rho_P:G_F\rightarrow
GL_2(h/P)$ unramified outside $pN$ such that

(1). $\mathrm{tr}(\rho_P)(Frob_\ell)=T_\ell$ for all prime ideals
$\ell\nmid pN$ and
$\mathrm{det}(\rho_P)=\varepsilon_+\mathcal{N}^n$ for the $p$-adic
cyclotomic character $\mathcal{N}$.

(2). Let $m$ be the unique maximal ideal containing $P$ and assume
that $T_{\mathbf{p}}\notin m$ for all primes $\mathbf{p}|p$, then
we have
$\rho_P|_{D_\mathbf{p}}\cong\begin{pmatrix}\epsilon_\mathbf{p}
&*\\0&\delta_\mathbf{p}\end{pmatrix}$ for the restriction of the
representation $\rho_P$ at the decomposition group $D_{\rho_P}$.
Moreover,
$\delta_\mathbf{p}([\varpi_\mathbf{p},F_\mathbf{p}])=U_\mathbf{p}(\varpi_\mathbf{p})$
and
$\delta_\mathbf{p}([u,F_\mathbf{p}])=\varepsilon_{1,\mathbf{p}}(u)u^{-k_{1,\mathbf{p}}}$
for $u\in O_\mathbf{p}^{\times}$.

(3). Write $N=N_0\mathfrak{c}(\varepsilon^-)$ and suppose that
$N_0$ is square free and prime to $\mathfrak{c}$. If $\ell$ is a
prime factor of $N_0$ which prime to $p$ and $\ell^2\nmid N$, then
$\rho_P|D_{\ell}\cong\begin{pmatrix}\epsilon_{\ell}&*\\0&\delta_{\ell}\end{pmatrix}$
such that $\delta_{\ell}([\varpi_\ell,F_\ell])=U_\ell$ and
$\delta_{\ell}([u,F_\ell])=\varepsilon_{1,\ell}(u)$ for $u\in
\mathcal{O}_\ell^{\times}$.
\end{theorem}

Here in the statement (2) of the above theorem, the symbol
$u^{-k_{1,\textbf{p}}}$ denote the product $\prod_{\sigma\in
I_\textbf{p}}\sigma(u)^{-k_{1,\sigma}}$ where $I_\textbf{p}$ is
the subset of $I$ consisting of the embedding $\sigma:F\rightarrow
\overline{\mathbb{Q}}\rightarrow\overline{\mathbb{Q}}_p$ which
give rise to the completion $F_\textbf{p}$. Since we choose
$W[\frac{1}{p}]$ contains all the image of the embedding of
$\mathcal{O}_F$, we may regard the above embedding actually into
$W[\frac{1}{p}]$.

\vskip 0.5cm

Let $m$ be a maximal ideal of $h$, then the localization $h_m$ is
a direct summand of $h$. The tensor product
$S_k(N,\varepsilon;A)_m=S_k(N,\varepsilon;A)\otimes_hh_m$ for a
$W$-algebra $A$ is then a direct summand of $S_k(N,\varepsilon;A)$
and can be identified with $\textrm{Hom}_A(h_m\otimes_WA,A)$.

If the representation $\overline{\rho}=\rho_m$ constructed as
above with the extra property that
$\overline{\varepsilon}_\ell\neq\overline{\delta}_\ell$, we call
this condition (ds) at $\ell$. From now on, we assume the square
free hypothesis in (3) of the above theorem.

Let $x$ be a prime ideal of $F$ such that $x\nmid pN$ and write
$\varpi_x$ the uniformizer of $x$. Consider the map
$$i_x:S_k(N,\varepsilon;A)^2\rightarrow S_k(Nx,\varepsilon;A)$$
sends $(f_1,f_2)\in S_k(N,\varepsilon;A)^2$ to $f_1+f_2|[\eta_x]$
for $\eta_x=\begin{pmatrix}1&0\\0&\varpi_x\end{pmatrix}$. We
denote $S_k(Nx,\varepsilon;A)_m^{old}$ the image of
$S_k(N,\varepsilon;A)_m^2$ under the map $i_x$ and denote
$h_m^{old}\subset\textrm{End}_A(S_k(Nx,\varepsilon;A)_m^{old})$
generated by all the Hecke operators.

\begin{lemma}
Define a linear operator $U$ acting on $S_k(N,\varepsilon;A)^2$ by
the matrix
$\begin{pmatrix}T_{x}&-1\\\varepsilon_+(\varpi_x)N(x)^n&0\end{pmatrix}$
multiplication on the right of $(f_1,f_2)\in
S_k(N,\varepsilon;A)^2$. Then we have $i_x\circ U=U_x\circ i_x$.
Furthermore, if the residue representation
$\overline{\rho}:G_F\rightarrow GL_2(h/m)$ associate to the
maximal ideal $m$ of the Hecke algebra $h=h_k(N,\varepsilon;W)$
restrict to the decomposition group $D_x$ satisfies the (ds)
condition, then there is an isomorphism
$$h_m^{old}\cong h_m[X]/(X^2-T_xX+\varepsilon_+(\varpi_x)N(x)^n)\cong h_m^2.$$
\end{lemma}

\begin{proof}
$$f|T_x=[\eta_x](f)+\sum_{a\bmod x}[\begin{pmatrix}1&a\\0&1\end{pmatrix}
\eta_x^{\iota}](f)=[\eta_x](f)+U_x(f).$$
which means that on $S_k(Nx,\varepsilon;A)_m^{Old}$, we have
$T_x=U_x+[\eta_x]$. On the other hand,
$[\eta_x^\iota\eta_x]=[\varpi_x]$, so
$U_x\circ[\eta_x]=\varepsilon_+(\varpi_x)N(x)^n$ on
$S_k(Nx,\varepsilon;A)_m^{Old}$.

\begin{align*}
U_xi_x((f_1,f_2))&=U_x(f_1+[\eta_x]f_2)=T_xf_1-[\eta_x]f_1+\varepsilon_+(\varpi_x)N(x)^nf_2\\
&=i_x((f_1,f_2)\begin{pmatrix}T_x&-1\\
\varepsilon_+(\varpi_x)N(x)^n&0\end{pmatrix})=i_x((f_1,f_2)U)
\end{align*}
for $(f_1,f_2)\in S_k(N,\varepsilon;A)_m^2$. The determinant
$\det(U)=\varepsilon_+(\varpi_x)N(x)^n\in W^{\times}$ is
invertible, which implies that $i_x$ is an isomorphism. $U_x$ is
invertible in $h_m^{old}$ and
$T_x=U_x+\varepsilon_+(\varpi_x)N(x)^nU_x^{-1}\in h_m^{old}$, thus
$h_m^{old}=h_m[U_x]\cong
h_m[Y]/(Y^2-T_xY+\varepsilon_+(\varpi_x)N(x)^n)$. If further
assume $\overline{\rho}$ satisfies (ds) condition at $x$, the
equation $Y^2-T_xY+\varepsilon_+(\varpi_x)N(x)^n=0$ has two
distinct roots modulo $m$, so has two distinct roots by Hensel
lemma, which implies
$h_m[Y]/(Y^2-T_xY+\varepsilon_+(\varpi_x)N(x)^n)\cong h_m^2$ as
algebras.
\end{proof}

Let $Q$ be a finite set of finite primes of $F$ such that
$\overline{\rho}$ satisfies (ds) at $v$ and $N(v)\equiv 1(\bmod
p)$ for all $v\in Q$. Let $NQ$ be the ideal $N\prod_{v\in
Q}(\varpi_v)$ and let $\widehat{\Gamma}_1^1(Q)=\prod_{v\in
Q}\widehat{\Gamma}_1^1(\varpi_v)$. For each $v\in Q$, choose a
solution $\alpha_v$ of the equation
$Y^2-T_vY+\varepsilon_+(\varpi_x)N(x)^n=0$. We thus has a maximal
ideal $m_Q$ of $h_k(NQ,\varepsilon;A)$ in the old part generated
by $\pi$,$T_x-\textrm{tr}\overline{\rho}(Frob_x)$ for $x\nmid pNQ$
and $U_v-\alpha_v$ for $v\in Q$. The above lemma implies
$h_k(NQ,\varepsilon;A)_{m_Q}\cong h_k(N,\varepsilon;A)_m$.

For each $v\in Q$, let $\Delta_v$ be the maximal $p$-power
quotient of $(\mathcal{O}_v/\varpi_v\mathcal{O}_v)^{\times}$, and
let $\Delta=\prod_{v\in Q}\Delta_v$. Define $U_v$ to be the kernel
of the projection
$(\mathcal{O}_v/\varpi_v\mathcal{O}_v)^{\times}\rightarrow\Delta_v$,
i.e, $U_v$ consists of elements in
$(\mathcal{O}_v/\varpi_v\mathcal{O}_v)^{\times}$ with order prime
to $p$. Denote by $S_k(N_Q,\varepsilon;A)$ the forms in
$S_k(\widehat{\Gamma}_0(N)\cap\widehat{\Gamma}_1^1(Q),\varepsilon;A)$
fixed by the action of $\prod_{v\in Q}U_v$. Let
$h_k(N_Q,\varepsilon;A)$ be the elements in
$\mathrm{End}_A(S_k(N_{Q},\varepsilon;A))$ generated by Hecke
operators. $h_k(N_Q,\varepsilon;W)$ is then a $W[\Delta]$-algebra
and $h_k(N_Q,\varepsilon;W)/(\delta-1,\delta\in\Delta)\cong
h_k(NQ,\varepsilon;W)$. Let $\mathfrak{m}_Q$ be the inverse image
of $m_Q$ under the projection $h_k(N_Q,\varepsilon;W)\rightarrow
h_k(NQ,\varepsilon;W)$, so $\mathfrak{m}_Q$ is a maximal ideal of
$h_k(N_Q,\varepsilon;W)$.

\begin{lemma}
The group $\Delta$ acting on
$S_k(N_Q,\varepsilon;A)_{\mathfrak{m}_Q}$ induces an isomorphism
between the invariants of the action, and the space
$S_k(N,\varepsilon;A)_m$. The localization
$h_k(N_Q,\varepsilon;W)_{\mathfrak{m}_Q}$ is a
$W[\Delta]$-algebra, free of finite rank over $W[\Delta]$.
\end{lemma}
For a proof of this lemma, see Hida \cite{HMF} section 3.2.3.

\vskip 0.5cm

\section{Galois Cohomology} \label{sec3}

Fix a finite extension $K$ of $\mathbb{Q}_p$ with integer ring
$W$, uniformiser $\pi_K$ and residue field $k$. Denote $CNL_W$ the
category of complete neotherian local $W$-algebras with residue
field $k$. Let $S$ be a finite set of primes of $F$ containing all
the primes dividing $p$. Denote the Galois group of the maximal
extension of $F$ unramified outside $S$ by $G_{F,S}$. Choose a
decomposition group $G_{F_v}$ for each $v\in S$ once and for all.
Let $\psi:(\mathbb{A}_{F}^f)^{\times}/F^{\times}\rightarrow
(W)^{\times}$ be a continuous character unramified outside $S$,
and we regard it as a Galois character via class field theory
\begin{equation}
\textrm{Gal}(\overline{F}^{ab}/F)\simeq(F\otimes_\mathbb{Q}\mathbb{R})
^{\times,+}\diagdown\mathbb{A}_F^{\times}\diagup
F^{\times}\rightarrow(\mathbb{A}_F^f)^{\times}\diagup F^{\times}.
\end{equation}

Let $\overline{\rho}:G_{F,S}\rightarrow GL_2(k)$ be a continuous
representation on a two dimensional $k$-vector space $V$. After
possibly replacing $k$ by a quadratic extension, we may and do
assume that $k$ contains all the eigenvalues of the image of
$\overline{\rho}$.

We will consider the framed deformations of representations for
both the decomposition groups and the global Galois group
$G_{F,S}$. For each prime $\textbf{p}_i$ of $F$ dividing $p$, we
choose a basis $\beta_k^i$ of the Galois module $k^2$, and
consider the functor from the category $CNL_W$ to the category of
sets, sending a $W$-algebra $A$ with residue field $k$, to the set
of equivalence classes of the pairs $(\rho_A,\beta_A)$, where
$\rho_A:G_{F_{\textbf{p}_i}}\rightarrow GL_2(A)$ is a deformation
of $\overline{\rho}|_{G_{F_{\textbf{p}_i}}}$ and $\beta_A$ becomes
$\beta_k^i$ via the isomorphism $A^2\otimes_AA/m_A\backsimeq k^2$.
We may regard $\beta_A$ as a two by two matrix with entries in $A$
which, after modulo the maximal ideal $m_A$, became the matrix
$\beta_k^i$ in $M_2(k)$. Two such pairs $(\rho_A,\beta_A)$ and
$(\rho_A',\beta_A')$ are equivalent, if there exist a matrix $T\in
GL_2(A)$ such that $T\equiv\textrm{id}(\bmod~m_A)$ and
$T(\rho_A)T^{-1}=\rho_A'$, $T\beta_A=\beta_A'$. This functor is
representable by an $W$-algebra in $CNL_W$, which we denoted by
$\mathcal{R}_{\textbf{p}}^{\square}$.

In the special case when $\overline{\rho}|_{G_{F_{\textbf{p}_i}}}$
is a scaler, this functor has another description. The set
$\textrm{Spec}\mathcal{R}_{\textbf{p}}^{\square}(A)$ can also be
regarded as $\{(\rho_A,\beta_A)\}/\sim$ for
$\rho_A:G_{F_{\textbf{p}_i}}\rightarrow GL_2(A)$ a lifting of
$\overline{\rho}|_{G_{F_{\textbf{p}_i}}}$ and $\beta_A$ an
arbitrary basis of the rank two module $A^2$ (not necessary a
lifting of $\beta_k^i$), while
$(\rho_A,\beta_A)\sim(\rho_A',\beta_A')$ if there exist a matrix
$T\in GL_2(A)$ (not necessary a lifting of the identity matrix)
such that $T(\rho_A)T^{-1}=\rho_A'$ and $T\beta_A=\beta_A'$. One
can check these two descriptions of the set
$\textrm{Spec}\mathcal{R}_{\textbf{p}_i}^{\square}(A)$ are the
same, if the representation
$\overline{\rho}|_{G_{F_{\textbf{p}}}}$ is a scaler. Both these
two descriptions are need.

Let $\mathcal{R}_{\textbf{p}}^{\square,\psi}$ be the quotient of
$\mathcal{R}_{\textbf{p}}^{\square}$ corresponding to the
deformations with fixed determinant $\psi$, and write
$\mathcal{R}_{p}^{\square}=\widehat{\otimes}_{\textbf{p}|p}
\mathcal{R}_{\textbf{p}}^{\square}$ and
$\mathcal{R}_{p}^{\square,\psi}=\widehat{\otimes}_{\textbf{p}|p}
\mathcal{R}_{\textbf{p}}^{\square,\psi}$.

Next we consider the functor from the category $CNL_W$ to the
category of sets sending $A$ to the set of equivalence classes of
the pairs $(\rho_A,\beta_A^i|_{\textrm{p}_i|p})$, consists of a
deformation $\rho_A:G_{F,S}\rightarrow GL_2(A)$ of
$\overline{\rho}$, and an $A$-basis $\beta_A^i$ lifting the chosen
basis $\beta_k^i$ for each $\textrm{p}_i|p$. Two pairs
$(\rho_A,\beta_A^i|_{\textrm{p}_i|p})$ and
$(\rho_A',{\beta_A'}^i)|_{\textrm{p}_i|p})$ are equivalent, if
thre exist a matrix $T\in GL_2(A)$ such that
$T\equiv\textrm{id}(\bmod~m_A)$, $T(\rho_A)T^{-1}=\rho_A'$ and
$T\beta_A^i={\beta_A'}^i$ for every $\textrm{p}_i|p$. This functor
is representable by some $W$-algebra
$\mathcal{R}_{F,S}^{\square}$. If $\overline{\rho}$ is absolutely
irreducible, the universal (non-framed) deformation functor is
also representable by a $W$-algebra $\mathcal{R}_{F,S}$.
Similarly, write $\mathcal{R}_{F,S}^{\square,\psi}$ for the
quotient of $\mathcal{R}_{F,S}^{\square}$ corresponding to
deformations with fixed determinant $\psi$.

\vskip 0.5cm

Denote by $\Sigma$ the places $\textbf{p}$ above $p$ and write
$r=|\Sigma|$. For any representation $(\rho, V)$ of a group
$G_{F,S}$ (resp. $G_{v}$), over a rank two free module $V$ over
some ring $A$, denote $ad(\rho)$ (resp. $ad^0(\rho)$) or $ad(V)$
(resp. $ad^0(V)$) the adjoint action of $G$ on $\textrm{End}(V)$
(resp. the trace zero elements in $\textrm{End}(V)$). We have the
following relation between the rings $\mathcal{R}_{F,S}^{\square}$
and $\mathcal{R}_{F,S}$.

\begin{proposition} \label{Prop3.1}
Assume $\overline{\rho}$ is absolutely irreducible. The morphism
$\mathcal{R}_{F,S}\rightarrow \mathcal{R}_{F,S}^{\square}$ is of
relative dimension $j=4|\Sigma|-1=4r-1$. The ring
$\mathcal{R}_{F,S}^{\square}$ can be identified to a power series
ring $\mathcal{R}_{F,S}[[w_1,\cdots,w_j]]$.
\end{proposition}

\begin{proof}
Since we assume that the residue representation $\overline{\rho}$
is absolutely irreducible, we have the tangent dimension
\begin{align*}
\dim(\mathcal{R}_{F,S}/(\pi_K)\mathcal{R}_{F,S})=&\dim
H^1(G_{F,S},ad(\overline{\rho}))-\dim
H^0(G_{F,S},ad(\overline{\rho}))\\
=&\dim H^1(G_{F,S},ad(\overline{\rho}))-1.
\end{align*}

These are computed in the following way. Each deformation is
regarded as $\rho:G_{F,S}\rightarrow GL_2(A)$, determined by a
matrix valued function $A(g)$ of the Galois group $G_{F,S}$. Let
$A(g)=A_0(g)+\varepsilon A_1(g)$, where $A_0(g)$ is the
representation $\overline{\rho}$. Then $A_1(g)A_0(g)^{-1}$ is a
cocycle in the module $ad(\overline{\rho})$. The deformation is
trivial, if there is a matrix $T=T_0+\varepsilon T_1$ such that
$TA(g)T^{-1}=A_0(g)$, ie, $T_0\in H^0(G_{F,S},ad(\rho_0))$ and
$T_1T_0^{-1}$ is a coboundary.

For the framed deformation, two pairs
$(\rho_1,\beta_1^1,\cdots,\beta_r^1)$ and
$(\rho_2,\beta_1^2,\cdots,\beta_r^2)$ are equivalent, if there is
some $T$ as above, such that $T\rho_1(g)T^{-1}=\rho_2(g)$ and
$T\beta_i^1=\beta_i^2$. Now assume we have a framed deformation
$(\rho_A,\beta_1,\cdots,\beta_r)$, we have a morphism
$R_{F,S}\rightarrow A$ such that the composition $G\rightarrow
GL_2(R_{F,S})\rightarrow GL_2(A)$ is equivalent to the given
$\rho_A$, say they are up to conjugation by $T$. Replace the
original pair $(\rho_A,\beta_1,\cdots,\beta_r)$ by $(T\rho_A
T^{-1},T\beta_1,\cdots,T\beta_r)$. Then there is no way to change
the $\rho_A$. However, a multiplication by a scaler matrix $T$
won't change the equivalence class of the framed deformation
pairs, and this is the only way to get an equivalent pair while
preserving the shape of the representation. Thus, we may also fix
one of the coordinate of the frames, say the right bottom corner
of $\beta_r$ have value $1$. Then we can see that the
$R_{F,S}[[w_1,\cdots,w_j]]$ has the universal property and the
uniqueness, while the $r$ pairs of basis are given by
$\beta_i=\begin{pmatrix}1+w_{4i-3}&w_{4i-2}\\w_{4i-1}&1+w_{4i}\end{pmatrix}$
for $i=1,2,\cdots,r$, where the last one $w_{4r}=0$ as we fixed.
\end{proof}

Restrict the universal representation $G_{F,S}\rightarrow
GL_2(\mathcal{R}_{F,S}^{\square,\psi})$ to the decomposition
groups, $\mathcal{R}_{F,S}^{\square,\psi}$ become an algebra over
$\mathcal{R}_{p}^{\square,\psi}$ by the universality. Both of
these two rings describe deformations without any local
conditions. To encode the local description, later we will
consider certain quotients of these rings. The following lemma of
Kisin \cite{Kis} described the relative tangential dimensions of
these two rings in terms of Galois cohomology.

\begin{lemma}(Kisin) Let $\delta_v=\dim_kH^0(G_{F_v},ad(\overline{\rho}))$
for $v|p$ and $\delta_F=\dim_kH^0(G_{F,S},ad(\overline{\rho}))$.
Define
$$H_{\Sigma}^1(G_{F,S},ad^0(\overline{\rho}))=\ker(\theta^1:H^1(G_{F,S},ad^0(\overline{\rho}))
\rightarrow\prod_{v|p}H^1(G_{F_v},ad^0(\overline{\rho}))).$$ Then
$\mathcal{R}_{F,S}^{\square,\psi}$ is a quotient of a power series
ring over $\mathcal{R}_{p}^{\square,\psi}$ in $g=\dim
H_{\Sigma}^1(G_{F,S},ad^0(\overline{\rho}))+\sum_{v|p}\delta_v-\delta_F$
variables.
\end{lemma}

From now on, we will make the following assumptions.

(1)\hskip 0.4cm $\overline{\rho}|_{G_{F_v}}$ is scaler for each
$v|p$.

(2)\hskip 0.4cm $\overline{\rho}$ has odd determinant and the
restriction to $G_{F(\zeta_p)}$ is absolutely irreducible.

(3)\hskip 0.4cm $p\geq5$. If $p=5$ and $\overline{\rho}$ has
projective image isomorphic to $PGL_2(\mathbb{F}_5)$, then
$[F(\zeta_p):F]=4$.

(4)\hskip 0.4cm For every $v\in S\diagdown \Sigma$, we have
\begin{equation}
(1-N(v))[(1+N(v))^2\det\overline{\rho}(Frob_v)-N(v)(\textrm{tr}\overline{\rho}(Frob_v))^2]\in
k^{\times}.
\end{equation}

Here $Frob_v$ means the arithmetic Frobenius at $v$. The scaler
condition (1) is the deformation problems we want to discuss in
this paper. This type of deformation problems are not covered by
Fujiwara \cite{Fuj} and Skinner-Wiles \cite{SW}, because the
deformation functor is no longer representable if the restriction
of $\overline{\rho}$ to decomposition groups are scaler. This is
the main problem we will solve in this paper by inventing a new
type of framed deformation functor. Assumptions (2) and (3) are
standard requirements for the Taylor-Wiles argument. (4) is to
simplify the argument of bad primes outside $p$, which may be
removed if one use Fujiwara's more general argument.

The Taylor-Wiles argument needs to enlarge the set of ramification
primes. The relevant result here we quote Kisin's version in
\cite{Kis} section 3.

\begin{proposition}(Kisin) Write
$g=\dim_kH^1(G_{F,S},ad^0\overline{\rho}(1))-[F:Q]+|\Sigma|-1$.
For each positive integer $n$, there exist a finite set of primes
$Q_n$ of $F$, disjoint from $S$, such that

(1)\hskip 0.4cm For each $v\in Q_n$, $N(v)\equiv 1(\bmod~p^n)$ and
$\overline{\rho}(Frob_v)$ has distinct eigenvalues.

(2)\hskip 0.4cm $|Q_n|=\dim_kH^1(G_{F,S},ad^0\overline{\rho}(1))$.

(3)\hskip 0.4cm Let $S_{Q_n}=S\cup Q_n$, then
$\mathcal{R}_{F,S_{Q_n}}^{\square,\psi}$ is topologically
generated over $\mathcal{R}_{p}^{\square,\psi}$ by $g$ elements.
\end{proposition}

For each $n$, fix a set of primes $Q_n$ and denote by
$\Delta_{Q_n}$ the Sylow $p$-subgroup of $(\mathcal{O}/\prod_{v\in
Q_n}\varpi_v\mathcal{O})^{\times}=\prod_{v\in
Q_n}(\mathcal{O}/\varpi_v\mathcal{O})^{\times}$. For $v\in Q_n$,
let $\xi_v\in \Delta_{Q_n}$ be the image of some fixed generator
of $(\mathcal{O}_{F_v}/\varpi_v)^{\times}$. Write
$h^1=|Q_n|=\dim_kH^1(G_{F,S},ad^0\overline{\rho}(1))$ and order
the elements of $Q_n$ as $v_1,v_2,\cdots,v_{h^1}$.
$W[\Delta_{Q_n}]$ becomes a quotient of $W[[y_1,\cdots,y_{h^1}]]$
by mapping $y_i$ to $\xi_{v_i}-1$, say $W[\Delta_{Q_n}]\cong
W[[y_1,\cdots,y_{h^1}]]/\mathrm{b}$ for some idea
$\mathrm{b}\subseteq
((y_1+1)^{p^n}-1,(y_2+1)^{p^n}-1,\cdots,(y_{h^1}+1)^{p^n}-1)$. As
explained in \cite{CDT}, the ring $\mathcal{R}_{F,S_{Q_n}}$ is a
$W[\Delta_{Q_n}]$-algebra, and
$\mathcal{R}_{F,S}\cong\mathcal{R}_{F,S_{Q_n}}/(y_1-1,y_2-1,\cdots,y_{h^1}-1)$
via the above isomorphism.

\vskip 0.5cm

\section{Framed deformation rings with local
conditions}\label{sec4}

Recall in section \ref{sec2}, if we take the set $S$ to be the
places divide the ideal $Np$, there is a continuous Galois
representation
$$\rho_{m}:G_{F,S}\rightarrow
GL_2(h_k(N,\varepsilon,W)_m)$$ such that for each place $v\notin
S$, the characteristic polynomial of $\rho_{m}(Frob_v)$ is given
by $X^2-T_vX+\varepsilon_+\mathcal{N}(\varpi_v)^n$. We take the
character $\psi=\varepsilon_+\mathcal{N}^n$.

Denote by $\overline{\rho}:G_{F,S}\rightarrow GL_2(k)$ the
representation obtained by reducing $\rho_{m}$ modulo $m$ and we
assume that $\overline{\rho}$ is absolutely irreducible. We
further assume the ideal $m$ is nearly ordinary at all the places
$\textbf{p}|p$. Thus by theorem \ref{ThmHida} in section
\ref{sec2},
$$\rho_{m}|_{G_{F_\textbf{p}}}\thicksim
\begin{pmatrix}\epsilon_\textbf{p}&*\\0&\delta_\textbf{p}\end{pmatrix},$$
where
$\delta_\textbf{p}(\varpi_\textbf{p})=U_\textbf{p}(\varpi_\textbf{p})$
and
$\delta_\textbf{p}([u,F_\textbf{p}])=\varepsilon_{1,\textbf{p}}u^{-k_{1,\textrm{p}}}$.

When the reduction $\overline{\rho}$ satisfies the
distinguishedness condition, i.e,
$\overline{\epsilon}_\textbf{p}\neq\overline{\delta}_\textbf{p}$,
the usual (non-framed) deformation functor is representable, and
the lifting problem has been studied by many authors, cf Wiles
\cite{W95}, Taylor-Wiles \cite{TW}, Fujiwara \cite{Fuj} and
Skinner-Wiles \cite{SW}. In this paper, we focus on the worst
situation that
$$\overline{\rho}|_{G_{F_\textbf{p}}}\thicksim
\begin{pmatrix}\overline{\chi}_\textbf{p}&0\\0&\overline{\chi}_\textbf{p}\end{pmatrix},$$
where
$\overline{\chi}_\textbf{p}=\overline{\epsilon}_\textbf{p}=\overline{\delta}_\textbf{p}$,
in which case the usual deformation functor is no longer
representable.

\vskip 0.5cm

In order to form a deformation ring describe these local
conditions, we choose a basis $\beta_k^i$ of the Galois module
$k^2$ for each $\textbf{p}_i|p$ as in section \ref{sec3}. Consider
the functor $F_{\textbf{p}_i}^{\vartriangle,\psi,s}$ from the
category $CNL_W$ to the category of sets, such that
$F_{\textbf{p}_i}^{\vartriangle,\psi,s}(A)$ is the set of
equivalence classes of the pairs $(\rho_A,\beta_A)$ consisting of
deformations $\rho:G_{F_{\textbf{p}_i}}\rightarrow GL_2(A)$ of
$\overline{\rho}|_{G_{F_{\textbf{p}_i}}}$ with fixed determinant
$\psi$, and an $A$-basis $\beta_A$ lifting the chosen $k$-basis
$\beta_k^i$, under which $\rho|_{G_{F_{\textbf{p}_i}}}$ is given
by
$\begin{pmatrix}\chi_{1,\textbf{p}_i}&*\\0&\chi_{2,\textbf{p}_i}\end{pmatrix}$
for some characters $\chi_{1,\textbf{p}_i}$ and
$\chi_{2,\textbf{p}_i}$ of $G_{F_{\textbf{p}_i}}$ such that
$\chi_{2,\textbf{p}_i}|_{I_{\textbf{p}_i}}=\delta_{\textbf{p}_i}$.
In other words, the functor we considered is a framed deformation
in a chosen Borel subgroup of $GL_2$.

Two pairs $(\rho,\beta_A)$ and $(\rho',\beta_A')$ are equivalent,
if there exist an upper triangular matrix $T\equiv
\textrm{id}_2(\bmod~m_A)$, such that $T\rho T^{-1}=\rho'$ and
$T\beta_A=\beta_A'$. This functor
$F_\textbf{p}^{\vartriangle,\psi,s}$ is representable by checking
Schlessinger criteria. Denote by
$\mathcal{R}_\textbf{p}^{\vartriangle,\psi,s}$ the $W$-algebra
representing this functor. The ring
$\mathcal{R}_\textbf{p}^{\vartriangle,\psi,s}$ is a quotient of
$\mathcal{R}_{\textbf{p}}^{\square,\psi}$. We define
$\mathcal{R}_{p}^{\vartriangle,\psi,s}=
\widehat{\otimes}_{\textbf{p}|p}\mathcal{R}_{\textbf{p}}^{\vartriangle,\psi,s}$.

The functor $F_{\textbf{p}_i}^{\vartriangle,\psi,s}$ also has
another description. The set
$F_{\textbf{p}_i}^{\vartriangle,\psi,s}(A)$ can be regarded as the
collection $\{(\rho_A,\beta_A)\}/\sim$ where $\beta_A$ is an
arbitrary basis of the free module $A^2$ (not necessary lifting
the chosen $\beta_k^i$), and $\rho_A$ is a deformation lifting
$\overline{\rho}$ such that it's of the shape
$\begin{pmatrix}\chi_{1,\textbf{p}_i}&*\\0&\chi_{2,\textbf{p}_i}\end{pmatrix}$
under the basis $\beta_A$, for some characters
$\chi_{1,\textbf{p}_i}$ and $\chi_{2,\textbf{p}_i}$ of
$G_{F_{\textbf{p}_i}}$ such that
$\chi_{2,\textbf{p}_i}|_{I_{\textbf{p}_i}}=\delta_{\textbf{p}_i}$.
To pairs $(\rho_A,\beta_A)\sim(\rho_A',\beta_A')$, if there exist
a matrix $T\in GL_2(A)$ (not necessarily a lifting of identity)
such that $T\rho_AT^{-1}=\rho_A'$.


\begin{proposition}
$\dim\mathcal{R}_\mathbf{p}^{\vartriangle,\psi,s}/m_W
\mathcal{R}_\mathbf{p}^{\vartriangle,\psi,s}
=1+\dim\mathrm{Hom}(G_{F_\mathbf{p}},k)=2+[F_\mathbf{p}:\mathbb{Q}_p]$.
\end{proposition}

\begin{proof}
The above dimension is equal to $\dim
F_{\textbf{p}}^{\vartriangle,\psi,s}(k[\varepsilon])$, where
$\varepsilon^2=0$. The restriction of $\chi_{2,\textbf{p}}$ on the
inertia subgroup $I_\textbf{p}$ is fixed, only the image of
Frobenius element has one dimensional deformation. The determinant
of $\rho$ is fixed, which implies
$\chi_{1,\textbf{p}}=\psi\chi_{2,\textbf{p}}^{-1}$ is uniquely
determined by $\chi_{2.\textbf{p}}$. For fixed
$\chi_{1,\textbf{p}}$ and $\chi_{2,\textbf{p}}$, a deformation is
given by an extension of $\chi_{2,\textbf{p}}$ by
$\chi_{1,\textbf{p}}$, ie, the vector space
$H^1(G_{F_\textbf{p}},k)=\textrm{Hom}(G_{F_\textbf{p}},k)$, which
is of dimension $[F_\textbf{p}:\mathbb{Q}_p]+1$. The allowed ways
of choosing basis are precisely the equivalence relation
$\thicksim$, thus, they have no more contribution.
\end{proof}

Write $T=h_k(N,\varepsilon,W)_m$ and write $\rho_T$ for
$\rho_{m}$. Take a basis $\beta_T^i$ of $T^2$ for each for each
$\textbf{p}|p$, under which $\rho_T|_{G_{F_\textbf{p}}}$ is upper
triangular. Then we can consider the framed deformation of the
pair $(\overline{\rho},\beta_k^i)_{i=1,2,\cdots,r}$, where we fix
the basis $\beta_k^i$ to be $\beta_T^i$ modulo $m_T$. Choose a
basis $\beta^i$ of $(\mathcal{R}_{F,S})^2$ which lifting
$\beta_T^i$ via the surjective homomorphism
$\mathcal{R}_{F,S}\rightarrow T$ by the universality. We normalize
the isomorphism
\begin{equation}
R_{F,S}^{\square}\simeq R_{F,S}[[w_1,\cdots,w_{4j-1}]]
\end{equation}
such that for each ring $A$ in the category $CLN_W$ and a
homomorphism of $W$-algebra $\iota:R_{F,S}^{\square}\rightarrow
A$, the corresponding element in the set
$\textrm{Spec}R_{F,S}^{\square}(A)$ is
$(\iota\circ\rho_{R_{F,S}},\iota(\begin{pmatrix}1+w_{4i-3}
&w_{4i-2}\\w_{4i-1}&w_{4i}\end{pmatrix}\cdot\beta^i))_{i=1,2,\cdots,r}$,
where again we agree on that $w_{4r}=0$ as in the proof of
proposition \ref{Prop3.1}.

\vskip 0.5cm

Define
$T^{\square}=T\widehat{\otimes}_{R_{F,S}}R_{F,S}^{\square}$, which
is isomorphic to a power series ring over $T$. The representation
$\rho_T$ gives a surjective homomorphism of $W$-algebra
$\theta:\mathcal{R}_{F,S}\rightarrow T$, hence a surjective
homomorphism $\theta:\mathcal{R}_{F,S}^{\square}\rightarrow
T^{\square}$. The basis determined by the mapping $\theta$ is thus
given by
$$\beta_{T^{\square}}^i=\begin{pmatrix}1+w_{4i-3}&w_{4i-2}\\
w_{4i-1}&1+w_{4i}\end{pmatrix}\cdot\beta_T^i$$ for each prime
$\textbf{p}_i$. Write $\rho_{T^{\square}}$ for the composition of
$\rho_T$ and the natural inclusion $T\rightarrow T^{\square}$. The
pair
$(\rho_{T^{\square}}|_{G_{F_\textbf{p}}},\beta_{T^{\square}}^i)$
is not in the set
$F_{\textbf{p}}^{\vartriangle,\psi,s}(T^{\square})$, since there
is a lower left corner in the basis $\beta_{T^{\square}}^i$.
Define a quotient $T^{\vartriangle}$ of $T^{\square}$ by
$T^{\square}/(w_{4i-1},i=1,2\cdots,j)$ via the isomorphism in
proposition \ref{Prop3.1}, which is also the ring removed the
variables $w_{4i-1}$ for $i=1,2,\cdots,r$. The projection
$T^{\square}$ to $T^{\vartriangle}$ push the pair
$(\rho_{T^{\square}}|_{G_{F_\textbf{p}}},\beta_{T^{\square}}^i)$
to the pair
$(\rho_{T^{\vartriangle}}|_{G_{F_\textbf{p}}},\beta_{T^{\vartriangle}}^i)$
where the representation $\rho_{T^{\vartriangle}}$ is again the
composition of $\rho_T$ and the natural inclusion $T\rightarrow
T^{\vartriangle}$ while $\beta_{T^{\vartriangle}}^i$, as the image
of $\beta_{T^{\square}}^i$, now becomes $\begin{pmatrix}1+w_{4i-3}&w_{4i-2}\\
0&1+w_{4i}\end{pmatrix}\cdot\beta^i$ since $w_{4i-1}=0$ in the
ring $T^{\vartriangle}$. By the universality, the composition of
the maps $\mathcal{R}_{F,S}^{\square}\rightarrow
T^{\square}\rightarrow T^{\vartriangle}$ then factors through the
tensor product
$\mathcal{R}^{\vartriangle}=\mathcal{R}_{F,S}^{\square,\psi}
\widehat{\otimes}_{\mathcal{R}_{p}^{\square,\psi}}\mathcal{R}_{p}^{\vartriangle,\psi,s}$,
which we denoted again by $\theta$. The map
$\mathcal{R}_{F,S}\rightarrow T$ is surjective, which implies that
the map $\mathcal{R}^{\vartriangle}\rightarrow T^{\vartriangle}$
is also surjective.

Similarly, replace $S$ by $S_{Q_n}$, we can define
$\mathcal{R}_n^{\vartriangle}$ to be the tensor product
$\mathcal{R}_{F,S_{Q_n}}^{\square,\psi}
\widehat{\otimes}_{\mathcal{R}_{p}^{\square,\psi}}\mathcal{R}_{p}^{\vartriangle,\psi,s}$.
Take $Q=Q_n$ in section \ref{sec2}, write
$T_n=S_k(N_{{Q_n}},\varepsilon,W)_{\mathfrak{m}_{Q_n}}$,
$T_n^{\square}=T_n\widehat{\otimes}_{\mathcal{R}_{F,S_{Q_n}}}\mathcal{R}_{F,S_{Q_n}}^{\square}$
and $T_n^{\vartriangle}=T_n^{\square}/(w_{4i-1},i=1,2\cdots,j)$,
then there is a surjection $R_n^{\vartriangle}\rightarrow
T_n^{\vartriangle}$ of $W[[y_1,\cdots,y_{h_1}]]$ algebras.

\vskip 0.5cm

\section{Patching} \label{sec5}
In the section we present a patching criterion of the Taylor-Wiles
system after Fujiwara and Kisin. We slightly modified Hida's
simplified proof \cite{H07} of this proposition.

\begin{proposition} \label{Thm5.1}
Let $B$ be a complete local noetherian $W$-algebra generated by
$d$ elements in the maximal ideal over $W$, i.e, $B$ is a quotient
of the power series ring $W[[z_1,z_2,\cdots,z_d]]$.
$\theta:R\rightarrow T$ is a surjective homomorphism of
$B$-algebras. $h$ and $j$ are two fixed non-negative integers.
Assume for each positive integer $n$, there exist two rings $R_n$
and $T_n$ fitting into a commutative diagram of $W$-algebras:

$$\begin{CD}
B[[x_1,\cdots,x_{h+j-d}]]@>\phi_n>>R_n    @>\theta_n>>   T_n\\
                         @.        @VVV                  @VVV\\
                         @.        R      @>\theta>>     T
\end{CD}$$
which satisfies the following conditions:

(1)\hskip 0.4cm The above diagram consists of surjective
$B$-algebra homomorphisms.

(2)\hskip 0.4cm $R_n$ is a
$\Lambda=W[[y_1,\cdots,y_h,t_1,\cdots,t_j]]$-algebra, and
$(y_1,\cdots,y_h)R_n=\ker(R_n\rightarrow
R)$,$(y_1,\cdots,y_h)T_n=\ker(T_n\rightarrow T)$.

(3)\hskip 0.4cm The kernel
$\ss_n=\ker(W[[y_1,\cdots,y_h,t_1,t_1,\cdots,t_j]]\rightarrow T_n)$
is contained in the ideal
$((1+y_1)^{p^n}-1,\cdots,(1+y_h)^{p^n}-1)$, and $T_n$ is finite free
over $\Lambda/\ss_n$. In particular, $T$ is finite free over
$\Lambda_f:=W[[t_1,\cdots,t_j]]$.

Then, $\theta$ is an isomorphism and $B\cong
W[[z_1,z_2,\cdots,z_d]]$.
\end{proposition}

\begin{proof}
Write
$s=\textrm{rank}_{\Lambda_f}T=\textrm{rank}_{\Lambda/\ss_n}T_n$
and $r_n=sn(h+j)p^n$. Consider the ideal
$$\mathfrak{c}_n=(m_W)^n+((1+y_1)^{p^n}-1,\cdots,(1+y_h)^{p^n}-1,t_1,\cdots,t_j)$$
of $\Lambda$. As the residue field $k$ of $W$ is finite of $q$
elements, we have
$|T_{n'}/\mathfrak{c}_nT_{n'}|=|T_n/\mathfrak{c}_nT_n|=q^{r_n}$ for
all $n'\geqslant n$ and so
$\textrm{length}_{R_n}T_n/\mathfrak{c}_nT_n<r_n$. Hence
$m_{R}^{r_n}T_n/\mathfrak{c}_nT_n=0$ and the composition
$\theta:R_{n'}\rightarrow T_{n'}\rightarrow
T_{n'}/\mathfrak{c}_nT_{n'}$ factors though the homomorphism
$\theta:R_{n'}/(\mathfrak{c}_n+m_{R_{n'}}^{r_n})\rightarrow
T_{n'}/\mathfrak{c}_nT_{n'}$. In particular, $\theta:R\rightarrow
T\rightarrow T/\mathfrak{c}_nT$ factors though the homomorphism
$\theta:R/(\mathfrak{c}_n+m_{R}^{r_n})\rightarrow T/\mathfrak{c}_nT$

Call $(D,A)$ a patching datum of level $n$, if there is a
commutative diagram consisting of surjective homomorphism of
$B$-algebras
$$\begin{CD}
B[[x_1,\cdots,x_{h+j-d}]]@>      >>D                               @>>>A\\
                         @.        @VVV                            @VVV\\
                         @.        R/(\mathfrak{c}_n+m_{R}^{r_n})  @>>>T/\mathfrak{c}_nT
\end{CD}$$
and $D$ is a $\Lambda/\mathfrak{c}_n$-algebra(thus also a
$\Lambda$-algebra) with the property $m_{D}^{r_n}=0$. Two patching
datum $(D,A)$ and $(D',A')$ are isomorphic, if there exist
isomorphisms $D\cong D'$ and $A\cong A'$ which commute with the two
diagrams. For a given level $n$ of patching datum, the order of $D$
and so as well as $A$ is bounded. Thus there are only finitely many
patching datum of level $n$. On the other hand, the above argument
provide a patching datum
$\theta:R_{n'}/(\mathfrak{c}_n+m_{R_{n'}}^{r_n})\rightarrow
T_{n'}/\mathfrak{c}_nT_{n'}$, for each integer $n'\geqslant n$. By
Dirichlet's drawer principle, there is an infinite subset $I$ of the
natural numbers $\mathbb{N}$, such that for any two integers $n<n'$
in $I$, the patching datum
$(\theta:R_{n'}/(\mathfrak{c}_n+m_{R_{n'}}^{r_n}),
T_{n'}/\mathfrak{c}_nT_{n'})$ is isomorphic to
$(\theta:R_{n}/(\mathfrak{c}_n+m_{R_{n}}^{r_n}),
T_{n}/\mathfrak{c}_nT_{n})$. Take a projective limit
$R_{\infty}=\underleftarrow{\lim}_{n\in
I}\theta:R_{n}/(\mathfrak{c}_n+m_{R_{n}}^{r_n})$ and
$T_{\infty}=\underleftarrow{\lim}_{n\in
I}\theta:T_{n}/(\mathfrak{c}_n+m_{T_{n}}^{r_n})$, then we have a
commutative diagram of surjective morphisms
$$\begin{CD}
B[[x_1,\cdots,x_{h+j-d}]]@>\phi_{\infty}>>R_{\infty}    @>\theta_{\infty}>>T_{\infty}\\
                         @.               @VVV          @VVV\\
                         @.               R             @>\theta>>         T
\end{CD}$$
$R_{\infty}$ becomes a $\Lambda$-algebra and $T_{\infty}$ is free of
rank $s$ over $\Lambda=\underleftarrow{\lim}_{n\in
I}\Lambda/\ss_n\Lambda$. The Krull dimension of $T_{\infty}$ is then
equal to $h+j+1$, the Krull dimension of $\Lambda$. However, we have
a chain of surjective homomorphisms
\begin{equation} \label{E;5.1}
W[[z_1,\cdots,z_d,x_1,\cdots,x_{h+j-d}]]\rightarrow
B[[x_1,\cdots,x_{h+j-d}]]\rightarrow R_{\infty}\rightarrow
T_{\infty}.
\end{equation}
$\textrm{Spec}(T_{\infty})$ becomes a closed sub scheme of
$\textrm{Spec}(W[[z_1,\cdots,z_d,x_1,\cdots,x_{h+j-d}]])$; it
can't be a proper sub scheme since the latter has the same
dimension $h+j+1$ as the former, so $T_{\infty}\cong
W[[z_1,\cdots,z_d,x_1,\cdots,x_{h+j-d}]]$ which forces all the
arrows in (\ref{E;5.1}) are isomorphisms. In particular, we have
$B\cong W[[z_1,\cdots,z_d]]$ by the fact that
$W[[x_1,\cdots,x_{h+j-d}]]$ is faithfully flat over $W$,
$R_{\infty}\cong T_{\infty}$ and $R\cong
R_{\infty}/(y_1,y_2,\cdots,y_h)R_{\infty}\cong
T_{\infty}/(y_1,y_2,\cdots,y_h)T_{\infty}\cong T$.
\end{proof}

In our case, $\mathcal{R}_{p}^{\vartriangle,\psi,s}$ is a
$W$-algebra of relative tangent dimension
$d=\sum_{\textbf{p}|p}(2+[F_\textbf{p}:\mathbb{Q}_p])=2|\Sigma|+[F:\mathbb{Q}]$.
In other words, $\mathcal{R}_{p}^{\vartriangle,\psi,s}$ is a
quotient of $W[[z_1,z_2\cdots,z_d]]$. Apply this patching lemma to
our setting that $B=\mathcal{R}_{p}^{\vartriangle,\psi,s}$,
$R=\mathcal{R}^{\vartriangle}$,
$R_n=\mathcal{R}_{Q_n}^{\vartriangle}$, $T=T^{\vartriangle}$ and
$T_n=T_{Q_n}^{\vartriangle}$, we get the following
\begin{theorem} \label{Thm5.2}
$\mathcal{R}^{\vartriangle}=T^{\vartriangle}$.
\end{theorem}

One byproduct of this Taylor-Wiles argument is the following
corollary, which seems not follow directly from the definition of
the deformation ring.
\begin{corollary}
$\mathcal{R}_{p}^{\vartriangle,\psi,s}\cong
W[[z_1,z_2\cdots,z_d]]$.
\end{corollary}

\vskip 0.5cm

\section{Modularity} \label{sec6}
Let $f$ be a cuspidal Hilbert modular eigenform over $F$ and let
$E_{f,\lambda}$ be the $\lambda$ completion of its coefficient
field for some place $\lambda|p$. Denote by $O_{f,\lambda}$ the
integer ring of $E_{f,\lambda}$, $\pi_{f,\lambda}$ the uniformiser
and $k$ the residue field. A Galois representation
$\rho_{f,\lambda}:G_{F,S}\rightarrow GL_2(O_{f,\lambda})$ is
attached to $f$, where $S$ is a finite set of places containing
the infinite places, the primes above $p$ and the primes where $f$
is ramified.

Let $E$ be a finite extension of $\mathbb{Q}_p$ with integer ring
$O_E$,uniformiser $\pi$ and residue field $k$. If a representation
$\rho:G_{F,S}\rightarrow GL_2(E)$ is equivalent to a some
$\rho_{f,\lambda}$, then we call it modular. After conjugation, we
may assume $\rho:G_{F,S}\rightarrow GL_2(O_E)$ and denote
$\overline{\rho}$ the composition $G_{F,S}\rightarrow
GL_2(O_E)\rightarrow GL_2(k)$. Under the assumption that
$\overline{\rho}$ is irreducible, then it's independent of the
choice of the element we used to conjugate the image of $\rho$.
Call $\rho$ residually modular if $\overline{\rho}\sim
\overline{\rho}_{f,\lambda}$ for some Hilbert eigenform $f$ over
$F$. We assume $E$ is large enough that it contains all the
conjugates of $F$ and its residue field $k$ contains all the
eigenvalues of the images of $\overline{\rho}$.

To prove the modular lifting theorem, we need more descriptions on
the shape of the restriction of $\rho$ to the decomposition groups
above $p$. Write $\beta_0=(e_1,e_2)$ for the standard basis
$e_1=^t(1,0)$ and $e_2=^t(0,1)$ of $k^2$. Recall that in section
\ref{sec4} when we define the framed deformation rings, we have
chosen the basis $\beta_k^i$ of $k^2$ for each $\textrm{p}_i$,
according to the representation $\rho_T:G_{F,S}\rightarrow
GL_2(T)$ into the Hecke algebra. A matrix $\overline{M}_i\in
GL_2(k)$ is uniquely determined by
$\overline{M}_i\cdot\beta_0=\beta_k^i$ for each $i$.

\begin{theorem} \label{Thm6.1}
Let $\rho:G_{F,S}\rightarrow GL_2(O_E)$ be a continuous
representation with the following conditions hold

(1)\hskip 0.4cm The determinant of
$\rho=\varepsilon_+\mathcal{N}^n$ is odd for some finite order
character $\varepsilon_+$ and positive integer $n$, where
$\mathcal{N}$ is the $p$-adic cyclotomic character.

(2)\hskip 0.4cm ${\rho}$ is nearly ordinary at all palace
$\mathbf{p}$ above $p$,
${\rho}|_{G_{\mathbf{p}_i}}=M_i\begin{pmatrix}\epsilon_{\mathbf{p}_i}&*\\
0&\delta_{\mathbf{p}_i}\end{pmatrix}M_i^{-1}$ for some matrix
$M_i\in GL_2(O)$ such that $M_i\equiv\overline{M}_i(\bmod(\pi))$,
where
$\delta_{\mathbf{p}}([u,F_{\mathbf{p}}])=\varepsilon_{1,\mathbf{p}}(u)u^{-k_{1,\mathbf{p}}}$
for $u\in\mathcal{O}_{\mathbf{p}}^{\times}$, $k_1\in
\mathbb{Z}[I]$ and $k_2=(n+1)I-k_1$ satisfied the condition
$k_2-k_1+I\geqslant 2I$.

(3)\hskip 0.4cm For every $\mathbf{p}$ above $p$,there is a
character $\overline{\chi}_\mathbf{p}:G_{F,S}\rightarrow
k^{\times}$ such that
$\overline{\epsilon}_{\mathbf{p}}=\overline{\delta}_{\mathbf{p}}=\overline{\chi}_\mathbf{p}$,
and the residue representation
$\overline{\rho}|_{G_\mathbf{p}}\sim\begin{pmatrix}\overline{\chi}_\mathbf{p}
&0\\0&\overline{\chi}_\mathbf{p}\end{pmatrix}$.

(4)\hskip 0.4cm The restriction of $\overline{\rho}$ to
$G_{F(\zeta_p)}$ is absolutely irreducible. If $p=5$ and
$\overline{\rho}$ has projective image isomorphic to
$PGL_2(\mathbb{F}_5)$, then $[F(\zeta_p):F]=4$.

(5)\hskip 0.4cm  the representation of $GL_2(F)$ corresponding to
$\rho|_{G_v}$ under local Langlands correspondence is not special
at any prime $v\nmid p$.

If $\overline{\rho}\sim \overline{\rho}_f$ for some automorphic
form $f$ of weight $k=(k_1,k_2)$, then $\rho$ is associated to an
automorphic form of weight $k$, too.
\end{theorem}

\begin{remark}
The deformation problem characterized by the local conditions (2)
and (3) is not covered by Fujiwara's work \cite{Fuj}, which he
assumes the distinguished condition on nearly ordinary primes. We
make the additional description of the matrices $M_i$'s in (2),
because under the situation (3), unlike the situation in
\cite{Fuj} when (ds) holds, only assuming $\rho$ lifting
$\overline{\rho}$ is not enough to make sure that $\rho$ is in the
category parameterized by our (framed) deformation rings.
\end{remark}

\begin{remark}
The condition (5) can be removed either by consider the minimal
lifting problem at such $v$ as Fujiwara\cite{Fuj}, or by
introducing automorphic forms on a suitable quaternion algebra as
in Kisin\cite{Kis}. Here we keep to assume it in order to make the
argument shorter.
\end{remark}

\begin{remark}
Under the circumstance
$\rho|_{G_{\mathbf{p}_i}}=\begin{pmatrix}\chi_{\mathbf{p}_i}&0\\0&
\chi_{\mathbf{p}_i}\end{pmatrix}$ for all $i$, the hypothesis on
the existence of $M_i$ is automatic. The locally split
representation is conjecturally to be CM by Greenberg, which is
proved by Ghate and Hida \cite{GH} under the assumption
$F=\mathbb{Q}$ and the Distinguishedness condition. Thus our
deformation covered the (non-distinguished) CM case as an example.
\end{remark}

\begin{proof}
Enlarge the coefficient field $E$ and the ring $W$ suitably so
that we may assume $W=O_E$. We first prove the theorem assuming
the following extra condition

(*)\hskip 0.5cm $f\in S_{k}(p,\varepsilon,W)$ for the level $N=p$
and the ``Neben''
$\varepsilon=(\varepsilon_1,\varepsilon_2,\varepsilon_+)$.

Let $m$ be the maximal ideal of $h_{k}(p,\varepsilon,W)$
associated to $f$. The local ring $T=h_{k}(p,\varepsilon,W)_m$ is
in the nearly ordinary part of $h_{k}(p,\varepsilon,W)$ and thus
reduced and generated by $T_{\ell}$ for $\ell\nmid p$. Take the
set $S$ of bad primes to be those above $p$ and the character
$\psi=\varepsilon_+\omega^n$.

Associated to each $\textrm{p}_i$ the basis
$\beta_E^i=M_i\cdot(e_1,e_2)$ of $O_E^2$, where $e_1=^t(1,0)$ and
$e_2=^t(0,1)$. The basis $\beta_E^i$ congruent to the universal
basis $\beta^i$ in section \ref{sec4}, so the pair
$(\rho,\beta_E^i)$ gives a homomorphism
$$R_{F,S}^{\square,\psi}\rightarrow O_E$$
by the universality, which factor through the map
$$\mathcal{R}^{\vartriangle}=R_{F,S}^{\square,\psi}
\widehat{\otimes}_{\mathcal{R}_{p}^{\square,\psi}}\mathcal{R}_{p}^{\vartriangle,\psi,s}
\rightarrow O_E,$$ since $\rho|_{G_{\textrm{p}_i}}$ is upper
triangular under the basis $\beta_E^i$. So we have a homomorphism
$$h_{k}(p,\varepsilon,W)_m=T\hookrightarrow T^{\vartriangle}\cong
\mathcal{R}^{\vartriangle}\rightarrow W$$ which, by the duality of
Hecke algebra and modular form, implies that $\rho$ is modular.

In the general case, apply the base change technic developed by
Skinner-Wiles \cite{SW}, (see also (3.5) of \cite{Kis}). As we
assumed that $\rho$ is non special for all primes $\ell$ outside
$p$, the only possibilities for $\rho|_{G_{\ell}}$ are either
principle or induced from a character of a quadratic extension of
$F$. Both of these two cases will become unramified after a
suitable base change to a totally real field $F'$. For the primes
in the level above $p$, we can put any powers on them as the Hecke
algebra is nearly ordinary, thus, this reduce to the case (*). The
local descriptions (2) and (3) of $\overline{\rho}$ are obviously
preserved under restriction to a subgroup. Then the above argument
implies $\rho|_{G_{F'}}$ is modular, so $\rho$ is modular by
descent.
\end{proof}

\section{Locally cyclotomic deformation} \label{sec7}
In this section, we will prove a family version of the lifting
theorem. Assume the initial weight $k_0=(0,I)$ in this section and
we will write $\varepsilon_0$ for the inial ``Neben'' instead of
using the letter $\varepsilon$. As before, we again assume that
$\rho:G_{F,S}\rightarrow GL_2(W)$ is a representation satisfies
all the assumptions of theorem \ref{Thm6.1} for the weight $k=k_0$
and ``Neben'' $\varepsilon=\varepsilon_0$.

We consider a new local framed deformation functor
$\Phi_{cyc,\textbf{p}}^{\vartriangle,\phi,s}$ sending a ring $A$
in the category $CNL_W$ to the set of equivalence classes of pairs
$(\rho_A,\beta_A)$ consisting of a deformation
$\rho_A:G_{F_{\textbf{p}}}\rightarrow GL_2(A)$ of
$\overline{\rho}|_{G_{F_{\textbf{p}}}}$ with fixed determinant
$\psi$, and an $A$-basis $\beta_A$ lifting the chosen $\beta_k$,
under which $\rho_A$ is given by
$\begin{pmatrix}\chi_{1,\textbf{p}}&*\\0&\chi_{2,\textbf{p}}\end{pmatrix}$
for some characters $\chi_{1,\textbf{p}}$ and
$\chi_{2,\textbf{p}}$ of $G_{F_{\textbf{p}}}$ such that the
character
$\chi_{2,\textbf{p}}|_{I_{\textbf{p}}}\varepsilon_{1,\textbf{p}}^{-1}$
factor through
$\textrm{Gal}(F_{\textbf{p}}^{ur}(\mu_{p^{\infty}})/F_{\textbf{p}}^{ur})$,
where $F_{\textbf{p}}^{ur}$ is the maximal unramified extension of
$F_{\textbf{p}}$. Two pairs
$(\rho_A,\beta_A)\thicksim(\rho'_A,\beta'_A)$ are equivalent, if
there exist an upper triangular matrix
$T\equiv\textrm{id}_2(\bmod~m_A)$, such that
$T\rho_AT^{-1}=\rho'_A$ and $T\beta_A=\beta'_A$. Notice that the
only difference between the above definition and the functor
$F_{\textbf{p}}^{\vartriangle,\phi,s}$ studied in section
\ref{sec4} is that in the latter functor, we require that
$\chi_{2,\textbf{p}}|_{I_{\textbf{p}}}\varepsilon_{1,\textbf{p}}^{-1}$
is trivial (remember that we have fix the initial weight
$k=k_0=(0,I)$).

The tangent dimension of the local cyclotomic deformation functor
$\Phi_{cyc,\textbf{p}}^{\vartriangle,\phi,s}$ is equal to
$3+[F_{\textbf{p}}:\mathbb{Q}_p]$, one more than that of
$F_{\textbf{p}}^{\vartriangle,\phi,s}$. Again by checking the
Schlessinger criterions,
$\Phi_{cyc,\textbf{p}}^{\vartriangle,\phi,s}$ is also
representable by some complete local noetherian $W$-algebra
$\mathcal{R}_{cyc,\textbf{p}}^{\vartriangle,\phi,s}$.

Write $\Gamma_{\textbf{p}}$ for the $p$-Sylow subgroup of
$\textrm{Gal}(F_{\textbf{p}}^{ur}(\mu_{p^{\infty}})/F_{\textbf{p}}^{ur})$,
which is embedded into $1+p\mathbb{Z}_p\subset\mathbb{Z}_p^\times$
via the $p$-adic cyclotomic character. It is isomorphic to
$\mathbb{Z}_p$ by choosing some generator $\gamma_{\textbf{p}}$.
Let $\Gamma_F=\prod_{\textbf{p}|p}\Gamma_{\textbf{p}}$. We then
have an isomorphism $W[[\Gamma_F]]\cong W[[X_1,X_2,\cdots,X_r]]$
by sending $\gamma_{\textbf{p}_i}$ to $1+X_i$, where we ordered
the primes $\textbf{p}$'s as in section \ref{sec3}.

The universal representation
$\rho_{cyc,\textbf{p}}:G_{F_{\textbf{p}}}\rightarrow
GL_2(\mathcal{R}_{cyc,\textbf{p}}^{\vartriangle,\phi,s})
\cong\begin{pmatrix}\chi_{1,\textbf{p}}&*\\0&\chi_{2,\textbf{p}}\end{pmatrix}$
with $\chi_{2,\textbf{p}}\equiv\overline{\chi}_{\textbf{p}}$
provide that $\chi_{2,\textbf{p}}\varepsilon_{1,\textbf{p}}^{-1}$
factors through $\Gamma_{\textbf{p}}$ and we get a
$W[[\Gamma_\textbf{p}]]$-algebra structure on
$\mathcal{R}_{cyc,\textbf{p}}^{\vartriangle,\phi,s}$ via the
character $
\chi_{2,\textbf{p}}|_{I_{\textbf{p}}}\varepsilon_{1,\textbf{p}}^{-1}:
\Gamma_F\rightarrow\mathcal{R}_{cyc,\textbf{p}}^{\vartriangle,\phi,s}$.
Define $\mathcal{R}_{p,cyc}^{\vartriangle,\psi,s}:
=\widehat{\otimes}_{\textbf{p}|p}\mathcal{R}_{cyc,\textbf{p}}^{\vartriangle,\phi,s}$.
$\mathcal{R}_{p,cyc}^{\vartriangle,\psi,s}$ is a
$W[[\Gamma_F]]$-algebra and there is an isomorphism
$\mathcal{R}_{cyc,p}^{\vartriangle,\phi,s}/(X_1,\cdots,X_r)
\cong\mathcal{R}_{p}^{\vartriangle,\phi,s}$.

\vskip 0.5cm

Identify $\widehat{\Gamma}_0(p^s)/\widehat{\Gamma}_1^1(p^s)$ with
$((\mathcal{O}/p^s\mathcal{O})^2)^{\times}$ by sending
$\begin{pmatrix}a&b\\c&d\end{pmatrix}\in\widehat{\Gamma}_0(p^s)/\widehat{\Gamma}_1^1(p^s)$
to $(a,d)\in((\mathcal{O}/p^s\mathcal{O})^2)^{\times}$. We have an
isomorphism $(\mathcal{O}/p^s\mathcal{O})^{\times}\cong
\prod_{\textbf{p}|p}(\mathcal{O}_{\textbf{p}}/p^s\mathcal{O}_{\textbf{p}})^{\times}$
and the local norm maps
$N_{\textbf{p}}:\mathcal{O}_{\textbf{p}}^{\times}\rightarrow
\mathbb{Z}_p^{\times}$ form a surjective homomorphism
$N_p=\prod_{\textbf{p}|p}N_{\textbf{p}}:
(\mathcal{O}_{\textbf{p}}/p^s\mathcal{O}_{\textbf{p}})^{\times}\rightarrow
\prod_{\textbf{p}|p}N_{\textbf{p}}(\mathbb{Z}/p^s\mathbb{Z})^{\times}$
for each integer $s>0$. Let $\widehat{\Gamma}_{cyc}(p^s)$ be the
subgroup of $GL_2(\mathbb{A}_F^f)$ such that
$\widehat{\Gamma}_1^1(p^s)\subset\widehat{\Gamma}_{cyc}(p^s)\subset\widehat{\Gamma}_0(p^s)$
and
$$\widehat{\Gamma}_{cyc}(p^s)/\widehat{\Gamma}_1^1(p^s)=\textrm{Ker}(N_p^2:
((\mathcal{O}/p^s\mathcal{O})^2)^{\times})\rightarrow
(\prod_{\textbf{p}|p}(\mathcal{O}_{\textbf{p}}/p^s\mathcal{O}_{\textbf{p}})^2)^{\times}.$$
Put
$\widehat{\Gamma}_s=\widehat{\Gamma}_{cyc}(p^s)\cap\widehat{\Gamma}_0(N)$,
then there is an inclusion
$S_{k_0}(\widehat{\Gamma}_n,\varepsilon_0,W)\rightarrow
S_{k_0}(\widehat{\Gamma}_m,\varepsilon_0,W)$ for each pair of
positive integers $m>n$, which is compatible with the Hecke
algebras action. Thus there is a surjective $W$-algebra
homomorphism
$S_{k_0}(\widehat{\Gamma}_m,\varepsilon_0,W)\rightarrow
S_{k_0}(\widehat{\Gamma}_m,\varepsilon_0,W)$ by restriction, and
we define the projective limit
$$\mathbf{h}_{cyc}^{n.ord}(N,\varepsilon_0;W[[\Gamma_F]]):
=\underleftarrow{\lim}_nh_{k_0}^{n.ord}(\widehat{\Gamma}_n,\varepsilon_0;W),$$
where for any level group $\widehat{\Gamma}$ and "Neben"
$\varepsilon$, the nearly ordinary Hecke algebra is defined by
$h_{k_0}^{n.ord}(\widehat{\Gamma},\varepsilon;W):=e_ph_{k_0}^{n.ord}
(\widehat{\Gamma},\varepsilon;W)$
for Hida's idempotent
$$e_p=\lim_{n\rightarrow\infty}(\prod_{\textbf{p}|p}U(\textbf{p}))^{n!}$$
where $U(\textbf{p})$ is the Hecke operator normalized as
\cite{HMF} section 3.1.2, page 168.

It's known that $h_{k_0}^{n.ord}(\widehat{\Gamma},\varepsilon;W)$
is a torsion free $W[[\Gamma_F]]$-module of finite type
(\cite{HMF},theorem 3.53), so
$\mathbf{h}_{cyc}^{n.ord}(N,\varepsilon;W[[\Gamma_F]])$ is a
$p$-profinite semilocal ring, i.e, a direct sum of the
localizations at its maximal ideals. Let $f_0$ be an Hecke eigen
form in $S_{k_0}(Np,\varepsilon,W)$ such that the associated
Galois representation $\rho_{f_0}$ modulo $m_W$ equal to
$\overline{\rho}$. Let $m$ be the maximal ideal of
$\mathbf{h}_{cyc}^{n.ord}(N,\varepsilon;W[[\Gamma_F]])$ given by
$f_0$, and denote by $T_{cyc}$ the localization
$\mathbf{h}_{cyc}^{n.ord}(N,\varepsilon;W[[\Gamma_F]])_m$. Then
there is Galois representation(\cite{HMF},Proposition 3.49)
$$\rho_{T_{cyc}}:\textrm{Gal}(\overline{F}/F)\rightarrow
GL_2(T_{cyc})$$ unramified outside $pN$ with the following
properties:

(1)$\textrm{tr}(\rho_{T_{cyc}}(Frob_\ell))=T_{\ell}$ for
$\ell\nmid pN$.

(2)$\det\rho_{T_{cyc}}=\varepsilon_+\mathcal{N}$.

(3)For each $\textbf{p}_i|p$, there is a basis $\beta_{T_{cyc}}^i$
of $T_{cyc}^2$, such that the pair
$(\rho_{T_{cyc}|_{G_{F_{\textbf{p}}}}},\beta_{T_{cyc}})$ is in
$\Phi_{cyc,\textbf{p}}^{\vartriangle,\phi,s}(T_{cyc})$, where we
take the character $\phi=\varepsilon_+\mathcal{N}$.

Define
$T_{cyc}^{\vartriangle}:=T_{cyc}[[w_1,\cdots,w_{4r-1}]]/(w_{4i-1},i=1,\cdots,r)$
and define
$\beta_{T_{cyc}^{\vartriangle}}^i=\begin{pmatrix}1+w_{4i-3}&w_{4i-2}\\
0&1+w_{4i}\end{pmatrix}\cdot\beta_{T_{cyc}}^i$, then the pair
$(\rho_{T_{cyc}^{\vartriangle}},\beta_{T_{cyc}^{\vartriangle}}^i)_{i=1,\cdots,r}$,
where $\rho_{T_{cyc}^{\vartriangle}}$ is the composition of
$\rho_{T_{cyc}}$ and the natural inclusion $T_{cyc}\hookrightarrow
T_{cyc}^{\vartriangle}$, gives a surjective homomorphism
$$\pi:\mathcal{R}_{cyc}^{\vartriangle}=R_{F,S}^{\square,\psi}
\widehat{\otimes}_{\mathcal{R}_{p}^{\square,\psi}}
\mathcal{R}_{cyc,p}^{\vartriangle,\psi,s}
\rightarrow T_{cyc}^{\vartriangle},$$ which we will prove to be an
isomorphism in the following theorem.

\vskip 0.5cm

Write $(k,\varepsilon)$ for $k\in\mathbb{Z}[I]^2$ and
$\varepsilon$ an arbitrary pairs of weight and "Neben", such that
$k_{i,\sigma}=k_{i,\sigma'}$ whenever both $\sigma,\sigma'\in
I_{\textbf{p}}$, which we will denote by $k_{i,\textbf{p}}$.
Define a prime ideal $P_{k,\varepsilon}\subset W[[\Gamma]]$ to be
the ideal generated by
$(\chi_{2,\textbf{p}}(\gamma)-\omega(\gamma)^{-k_{1,\textbf{p}}}\varepsilon_1(\gamma))$
for $\textbf{p}|p$ and $\gamma\in\Gamma_{\textbf{p}}$.

\begin{theorem}
The surjective $W[[\Gamma_F]]$-algebra homomorphism
$\mathcal{R}_{cyc}^{\vartriangle}\rightarrow
T_{cyc}^{\vartriangle}$ is an isomorphism.
$\mathcal{R}_{cyc}^{\vartriangle}$ is free of finite rank over
$W[[\Gamma_F]][[w_1,\cdots,w_{4r-1}]]/(w_{4i-1},i=1,\cdots,r)$.
For any locally cyclotomic $(k,\varepsilon)$ such that

(1)\hskip 0.5cm $k_2-k_1\geqslant I$,

(2)\hskip 0.5cm
$z_{\mathbf{p}}^{-k_{j,\mathbf{p}}}\varepsilon_{j}(z_{\mathbf{p}})
=z_{\mathbf{p}}^{-k_{0,j,\mathbf{p}}}\varepsilon_{0,j}(z_{\mathbf{p}})$,

(3)\hskip 0.5cm $\varepsilon_+=\varepsilon_{0,+}$ and
$\varepsilon_j|_{\mathcal{O}_\ell^{\times}}=\varepsilon_{0,j}|_{\mathcal{O}_\ell^{\times}}$
for $j=1,2$,

(4)\hskip 0.5cm
$\mathfrak{c}(\varepsilon_{0,j})|\mathfrak{c}(\varepsilon_j)|p^\infty\mathfrak{c}
(\varepsilon_{0,j})$ for $j=1,2$.

Then there is a unique local factor $T_{k,\varepsilon}\subset
h_{k}^{n.ord}(N\cap\mathfrak{c}(\varepsilon^-),\varepsilon;W)$
such that we have the isomorphisms
$\mathcal{R}_{cyc}^{\vartriangle}/P_{k,\varepsilon}\mathcal{R}_{cyc}^{\vartriangle}\cong
T_{k,\varepsilon}^{\vartriangle}$ induced by $\pi$.
\end{theorem}

\begin{proof}
$\mathcal{R}_{cyc,\textbf{p}}^{\phi,\vartriangle,s}/P_{k,\varepsilon}
\mathcal{R}_{cyc,\textbf{p}}^{\phi,\vartriangle,s}$
is the maximal quotient of
$\mathcal{R}_{cyc,\textbf{p}}^{\phi,\vartriangle,s}$ on which
$\delta_{\textbf{p}}$ induces $[u,F_{\textbf{p}}]\rightarrow
N_{\textbf{p}}^{-k_{1,\textbf{p}}}\varepsilon_{1,\textbf{p}}(u)$
for $u\in \mathcal{O}_{\textbf{p}}$. Thus
$\mathcal{R}_{cyc,\textbf{p}}^{\phi,\vartriangle,s}/
P_{k,\varepsilon}\mathcal{R}_{cyc,\textbf{p}}^{\phi,\vartriangle,s}$
is the universal framed deformation ring of
$\overline{\rho}|_{G_{F_{\textbf{p}}}}$ for the deformations of
type $(k,\varepsilon)$ instead of $(k_0,\varepsilon_0)$. The
previous proved proposition \ref{Thm5.1} states that
$\mathcal{R}_{cyc}^{\vartriangle}/P_{k_0,\varepsilon_0}\mathcal{R}_{cyc}^{\vartriangle}
\cong\mathcal{R}^{\vartriangle}\cong T^{\vartriangle}$. Write $a$
for the rank of $T$ over $W$, there is a surjection
$W[[\Gamma_F]]^a\rightarrow T_{cyc}$ which induces $W^a\cong T$ by
modulo $P_{k_0,\varepsilon_0}$ by Nakayama's lemma. Since
$T_{cyc}$ is torsion free $W[[\Gamma_F]]$-module, the kernel must
be trivial, and we have $T_{cyc}\cong W[[\Gamma_F]]^a$. As
$\mathcal{R}_{cyc}^{\vartriangle}/P_{k_0,\varepsilon_0}\mathcal{R}_{cyc}^{\vartriangle}
\cong T^{\vartriangle}$, $\mathcal{R}_{cyc}^{\vartriangle}$ is
generated by at most $a$ elements as a
$W[[\Gamma_F]][[w_1,\cdots,w_{4r-1}]]/(w_{4i-1},i=1,\cdots,r)$-module.
In addition, $\mathcal{R}_{cyc}^{\vartriangle}$ surjectively cover
the rank $a$ free
$W[[\Gamma_F]][[w_1,\cdots,w_{4r-1}]]/(w_{4i-1},i=1,\cdots,r)$
module $T_{cyc}^{\vartriangle}$, we must have
$\mathcal{R}_{cyc}^{\vartriangle}\cong T_{cyc}^{\vartriangle}$.
\end{proof}

\vskip 0.5cm

\section{Characteristic zero deformation and $\mathcal{L}$-invariant}
\label{sec8}

There are many different ways to define $\mathcal{L}$-invariant,
which is expected to measure the difference between $p$-adic
$L$-function at an exceptional zero and the archimedean
$L$-function. The $\mathcal{L}$-invariant of an elliptic curve
$E_{/\mathbb{Q}}$ is studied by Mazur-Tate-Teitelbaum \cite{MTT}
and Greenberg-Stevens \cite{GS}. Hida studied the $L$-invariants
of Tate curves \cite{H07a} \cite{H07b} and make a vast
generalization of Mazur-Tate-Teitelbaum conjecture to the
symmetric powers of Galois representations associated to Hilbert
modular forms \cite{H07d}, which gave explicit predictions of the
$\mathcal{L}$-invariants of these representations. Hida proved
certain cases of this conjecture, assuming a conjectural shape of
certain universal deformation rings over characteristic zero
fields.

Let $K$ be the fraction field of $W$ and $\rho:G_{F,S}\rightarrow
GL_2(W)\rightarrow GL_2(K)$ is the representation associated to
some nearly ordinary Hilbert modular form $f$. We further assume
that $\epsilon_{\textbf{p}}\neq\delta_{\textbf{p}}$ for all
$\textbf{p}|p$, i.e, the two characters $\epsilon_{\textbf{p}}$
and $\delta_{\textbf{p}}$ of $G_{F_{\textbf{p}}}$ are distinct
over $W$ for each $\textbf{p}$, but become equal after modulo
$m_W$.

In \cite{HMF} section 3.2.9, Hida considered the following type of
deformation functor $\Phi_K$ from $ART_K$, the category of
artinian local $K$-algebra with residue field $K$, to the category
of sets, such that for each $A$ in $ART_K$, $\Phi_K(A)$ is the set
of isomorphism classes of deformations $\rho_A:G_{F,S}\rightarrow
GL_2(A)$ of $\rho_0$ satisfying the additional conditions:

(1) $\det(\rho_A)=\varepsilon_+\mathcal{N}$.

(2) $\rho_A|_{G_{F_{\textbf{p}}}}\thicksim
\begin{pmatrix}\epsilon_{A,\textbf{p}}&*\\0&\delta_{A,\textbf{p}}\end{pmatrix}$
for characters
$\epsilon_{A,\textbf{p}},\delta_{A,\textbf{p}}:D_{\textbf{p}}\rightarrow
A^{\times}$, such that $\delta_{A,\textbf{p}}\cong
\delta_{\textbf{p}}(\bmod~m_A)$ and
$\delta_{A,\textbf{p}}\delta_{\textbf{p}}^{-1}|_{I_{\textbf{p}}}$
factor through
$\mathrm{Gal}(F_{\textbf{p}}^{ur}(\mu_{p^\infty})/F_{\textbf{p}}^{ur})$
for all $\textbf{p}|p$.

Since $\rho$ is absolutely irreducible, the functor $\Phi_F^K(A)$
is pro-representable by a pro-artinian local $K$-algebra
$\mathcal{R}_K$ with residue field $K$. We denote the universal
representation $G_{F,S}\rightarrow GL_2(\mathcal{R}_K)$ by
$\rho_{\mathcal{R}_K}$. The restriction of
$\rho_{\mathcal{R}_K}|_{G_{F_{\mathbf{p}}}}$ is isomorphic to
$\begin{pmatrix}\epsilon_{\mathcal{R}_K,\mathbf{p}}&*\\0&
\delta_{\mathcal{R}_K,\mathbf{p}}\end{pmatrix}$.

Again we denote $\Gamma_{\textbf{p}}$ the $p$-Sylow subgroup of
$\textrm{Gal}(F_{\textbf{p}}^{ur}(\mu_{p^\infty})/F_{\textbf{p}})$,
which is isomorphic to a finite index subgroup of
$1+p\mathbb{Z}_p$. Denote by $\gamma_{\textbf{p}}$ the chose
generator as in the last section. The character
$\delta_{\mathcal{R}_K,\textbf{p}}\delta_{\textbf{p}}^{-1}$ make
the ring $\mathcal{R}_K$ into an $K[[t_{\textbf{p}}]]$-algebra, by
sending $(1+t_{\textbf{p}})$ to
$\delta_{\mathcal{R}_K,\textbf{p}}\delta_{\textbf{p}}^{-1}(\gamma_{\textbf{p}})$.
The ring $\mathcal{R}_K$ then become a
$K[[t_{\textbf{p}}]]|_{\textbf{p}|p}$-algebra.

\begin{conjecture}(Hida,\cite{H07d})The ring $\mathcal{R}_K$ is
isomorphic to $K[[t_{\mathbf{p}}]]|_{\mathbf{p}|p}$.
\end{conjecture}

In the case of representations associated to nearly ordinary
Hilbert modular forms with Distinguishedness (ds) condition, i.e,
if $\epsilon_{\textbf{p}}\neq\delta_{\textbf{p}}(\bmod~m_w)$,
Hida's conjecture can be deduced from the family version of
Fujiwara's $R=T$ theorem \cite{Fuj}, see Chapter 3 section 2 in
\cite{HMF}.

In this paper, as we always assume
$\epsilon_{\textbf{p}}\equiv\delta_{\textbf{p}}(\bmod~m_w)$,
Fujiwara's $R=T$ theorem fails, since the (ds) condition no long
hold and the universal deformation ring $R$ doesn't exist. On the
other hand, we have proved an isomorphism
$\mathcal{R}_{cyc}^{\vartriangle}\cong T_{cyc}^{\vartriangle}$ in
section \ref{sec7}. In this section, we will show that this
isomorphism between framed deformation ring and framed Hecke
algebra is enough to deduce Hida's conjecture, and thus get the
formula of $\mathcal{L}$-invariants.

For an arbitrary pro-Artinian local $W$-algebra $A$ and a
deformation $(\rho_A,\beta_A^1,\cdots,\beta_A^r)$ corresponding to
the homomorphism
$\theta:\mathcal{R}_{cyc}^{\vartriangle}\rightarrow A$, we can
modified the frames to $\beta_A^{'1},\cdots,\beta_A^{'r}$ such
that the kernel of the homomorphism
$\theta':\mathcal{R}_{cyc}^{\vartriangle}\rightarrow A$
corresponding to the deformation
$(\rho_A,\beta_A^{'1},\cdots,\beta_A^{'r})$ contains the ideal
$(w_1,w_2,\cdots,w_{4r-1})$. This can be done by taking
$\beta_A^{'i}=\begin{pmatrix}1+\theta(w_{4i-3})&\theta(w_{4i-2})
\\0&1+\theta(w_{4i})\end{pmatrix}^{-1}\cdot\beta_A^i$.

As we have obtained a homomorphism $T_{cyc}\rightarrow
T\rightarrow W$, the mapping
$\theta_0:\mathcal{R}_{cyc}^{\vartriangle}\rightarrow W$ sends
$w_1,\cdots,w_{4r-1}$ to zero. Denote $P$ the kernel of
$\theta_0$, we can see $w_i$'s and $(\gamma-1)$ for
$\gamma\in\Gamma_F$ are in $P$. The localization completion
$\widehat{\mathcal{R}}_{P}^{\vartriangle}:=\underleftarrow{\lim}_n(
(\mathcal{R}_{cyc}^{\vartriangle})_P/P^n)$ is a pro-Artinian local
$K$-algebra and $\rho_P:G_F\rightarrow
GL_2(\widehat{\mathcal{R}}_{P}^{\vartriangle})$ is obviously in
the set $\Phi_K(\widehat{\mathcal{R}}_{P}^{\vartriangle})$. There
is a $W[[\Gamma_F]]$-algebra homomorphism
$\vartheta:\mathcal{R}_K\rightarrow
\widehat{\mathcal{R}}_{P}^{\vartriangle}$ by the universality. To
describe the image of this map, we have the following theorem.

\begin{theorem} \label{Thm8.2}
Let $\pi$ be the projection
$\widehat{\mathcal{R}}_{P}^{\vartriangle}\rightarrow
\widehat{\mathcal{R}}_{P}^{\vartriangle}/(w_1,\cdots,w_{4r-1})$.
The composition $\pi\circ\vartheta:
\mathcal{R}_K\rightarrow\widehat{\mathcal{R}}_{P}^{\vartriangle}
/(w_1,\cdots,w_{4r-1})$ is an isomorphism.
\end{theorem}

\begin{proof}
It's well known that the ring $\mathcal{R}_{F,S}$ and
$\mathcal{R}_{K}$ are topologically generated by the trace of the
universal Galois representations, for example, page 244 of
\cite{HMF}. The framed ring $\mathcal{R}_{cyc}^{\vartriangle}$ and
hence the completion localization
$\widehat{\mathcal{R}}_{P}^{\vartriangle}$ are then generated by
the trace of image the universal representation, together with the
images of $w_i$'s. In particular, the quotient
$\widehat{\mathcal{R}}_{P}^{\vartriangle}/(w_1,\cdots,w_{4r-1})$
is generated only by the trace of $\rho_P$ too, and $\pi\circ
\vartheta$ is surjective. We prove the ring
$\widehat{\mathcal{R}}_{P}^{\vartriangle}/(w_1,\cdots,w_{4r-1})$
represents the functor $\Phi_F^K$ directly.

\vskip 0.3cm

For a local artinian $K$-algebra $A$ and a representation
$\rho_A\in \Phi_F^K(A)$, there is a $W$-lattice $L$ in the finite
dimensional $K$-vector space $A^2$ stable under $\rho_A(G_F)$. The
$W$-algebra $A_0=A\cap\textrm{End}_W(L)$ is compact and contains
the trace of the image of $\rho_A$. $A_0$ is a local $W$-algebra
free of finite rank over $W$ with maximal ideal $m_{A_0}=m_A\cap
A_0$. One can construct a representation
$\rho_{A_0}:G_F\rightarrow GL_2(A_0)$ by means of pseudo
representation which, after composing with the inclusion
$A_0\hookrightarrow A$, is isomorphic to $\rho_A$. For each
$\textbf{p}_i|p$, the two distinct characters
$\epsilon_{\textbf{p}}$ and $\delta_{\textbf{p}}$ having values in
$A_0$, so the local representation
$\rho_{A_0}|_{G_{\textbf{p}_i}}$ is isomorphic to a representation
into upper-triangular matrices over $GL_2(A_0)$. The reduction of
$\rho_{A_0}$ modulo $m_{A_0}$ is isomorphic to $\rho_0$. Choose a
basis $\beta_{A_0}^i$ such that
$(\rho_{A_0}|_{G_{\textbf{p}_i}},\beta_{A_0}^i)$ is in the set
$\Phi_{cyc,\textbf{p}}^{\vartriangle,\phi,s}$. The universality
gives a homomorphism
$\theta_0:\mathcal{R}_{cyc}^{\vartriangle}\rightarrow A_0$
corresponding to the pair
$(\rho_{A_0},\beta_{A_0}^1,\cdots,\beta_{A_0}^r)$. We can further
assume the $\theta_0$ maps $w_i$ to $0$ by the remark before this
theorem. The composition
$\mathcal{R}_{cyc}^{\vartriangle}\rightarrow A_0\hookrightarrow
A$, which we again denoted by $\theta_0$, factors through the
localization completion
$\widehat{\mathcal{R}}_{P}^{\vartriangle}$, and further factor
$\widehat{\mathcal{R}}_{P}^{\vartriangle}/(w_1,\cdots,w_{4r-1})$.
$\widehat{\mathcal{R}}_{P}^{\vartriangle}/(w_1,\cdots,w_{4r-1})$
must be universal, and $\pi\circ \vartheta$ is an isomorphism.
\end{proof}

\begin{theorem} \label{Thm8.3}
If we normalize the isomorphism $W[[\Gamma_F]]\cong
W[[X_\textbf{p}]]_{\mathbf{p}|p}$ sending the generators
$\gamma_{\mathbf{p}}$ to $1+X_\mathbf{p}$. Then Hida's conjecture
hold, i.e, there is an isomorphism $\mathcal{R}_K\cong
K[[t_{\mathbf{p}}]]|_{\mathbf{p}|p}$ such that
$t_{\mathbf{p}}=X_{\mathbf{p}}-p$.
\end{theorem}

\begin{proof}
By the above theorem, we have the isomorphisms $\mathcal{R}_K\cong
\widehat{\mathcal{R}}_P^{\vartriangle}/(w_1,\cdots,w_{4r-1})\cong
\widehat{T}_{cyc,P_0}$, where $P_0$ is the image of
$P/(w_1,\cdots,w_{4r-1})$ under the isomorphism
$\mathcal{R}_{cyc}^\vartriangle$ and
$\widehat{T}_{cyc,P_0}:=\underleftarrow{\lim}_n(T_{cyc,P_0}/P_0^n)$
is the localization completion of $T_{cyc}$ at $P_0$.

The Hecke ring $T_{cyc}/P_0T_{cyc}=T$, the local Hecek algebra for
forms of weight $k_0=(0,I)$ and Neben $\varepsilon_0$, is reduced,
since the level $N$ is square free. $T\otimes_WK$ is reduced and
so unramified and \'{e}tale over $K$. By proposition 3.8 of
charpter I in \cite{Mil}, $T_{cyc}$ is \'{e}tale over
$W[[\Gamma]]$ in an open neighborhood of $P_0$. The pull back of
$P_0$ to $W[[\Gamma]]$ is the prime ideal
$(X_{\textbf{p}}-p)_{\textbf{p}|p}$. The \'{e}taleness implies
that $\widehat{T}_{cyc,P_0}$ coincide to the completion
localization of $W[[\Gamma]]$ at
$(x_{\textbf{p}}-p)_{\textbf{p}|p}$ (see Theorem 4.2 in the
chapter I of \cite{Mil}), which is isomorphic to
$\mathcal{R}_K\cong K[[t_{\mathbf{p}}]]|_{\mathbf{p}|p}$ via the
map $t_{\textbf{p}}=X_{\textbf{p}}-p$.
\end{proof}

Ordered the primes $\mathbf{p}|p$ such that
$\rho|_{G_{F_{\mathbf{p}_i}}}\cong\begin{pmatrix}w&\xi_{q_i}\\0&1\end{pmatrix}
\otimes\delta_{\mathbf{p}_i}$ for $i\leq b$ and
$\epsilon_{\mathbf{p}_i}/\delta_{\mathbf{p}_i}\neq\omega$ for
$b<i\leq r$. A cocycle
$\xi_{q}:\mathrm{Gal}(\overline{F}_{\mathbf{p}}/F_{\mathbf{p}})\rightarrow
K(1)$ labelled by $q\in F_{\mathbf{p}}^\times$ is given by
$\xi_q=\underleftarrow{\lim}_n\xi_{q,n}$ for
$\xi_{q,n}(\sigma)=(q^{1/q^n})^{\sigma-1}$. Put
$Q_j=N_{F_{\mathbf{p}_j}/\mathbb{Q}_p}(q_j)$,
$F_i=F_{\textbf{p}_i}$ and
$\delta_{\mathcal{R}_K,i}=\delta_{\mathcal{R}_K,\textbf{p}_i}$.
Theorem \ref{Thm8.3} have the following corollary, by Hida's
Theorem 0.3 in \cite{H07d}.
\begin{corollary}  \label{Cor8.4}
The Greenberg $\mathcal{L}$-invariant of
$\mathrm{Ind}_F^{\mathbb{Q}}(ad^0(\rho))$ is given by
\begin{equation}
\det(\frac{\partial\delta_{\mathcal{R}_K,i}([p,F_i])}{\partial
t_j})_{i,j>b}|_{t_1=\cdots=t_r=0}\prod_{i>b}
\frac{\log_p(\gamma_i)}{\delta_{i}([p,F_i])}
\prod_{i=1}^b\frac{\log_p(Q_i)}{\mathrm{ord}_p(Q_i)}.
\end{equation}
\end{corollary}

\end{document}